\documentclass[suppldata]{gOMS2e}   

\usepackage{subfigure}
\usepackage{booktabs}
\usepackage{tikz}

\newtheorem{theorem}{Theorem}[section]

\makeatletter
\newcommand\MkNewTheorem[2]{%
  \newtheorem{#1}{#2}
  \expandafter\def\csname c@#1\endcsname{\c@theorem}
  \expandafter\def\csname p@#1\endcsname{\p@theorem}
  \expandafter\def\csname the#1\endcsname{\thetheorem}
  \expandafter\def\csname #1name\endcsname{#2}
}

\MkNewTheorem{corollary}{Corollary}
\MkNewTheorem{lemma}{Lemma}
\MkNewTheorem{proposition}{Proposition}
\MkNewTheorem{prop}{Proposition}
\MkNewTheorem{claim}{Claim}
\MkNewTheorem{observation}{Observation}
\MkNewTheorem{obs}{Observation}
\MkNewTheorem{conjecture}{Conjecture}
\MkNewTheorem{openquestion}{Open question}

\theoremstyle{definition}
\MkNewTheorem{example}{Example}
\MkNewTheorem{exercise}{Exercise}
\MkNewTheorem{notation}{Notation}
\MkNewTheorem{assumption}{Assumption}
\MkNewTheorem{definition}{Definition}
\MkNewTheorem{remark}{Remark}
\MkNewTheorem{goal}{Goal}
\MkNewTheorem{problem}{Problem}

\usepackage{verbatim}

\usepackage{tikz}
\definecolor{mediumspringgreen}{rgb}{0.0, 0.98039215, 0.60392156}

\newcommand{\bb}{\mathbb}
\newcommand{\R}{\bb R}
\newcommand{\Q}{\bb Q}
\newcommand{\Z}{\bb Z}

\newcommand{\B}{B}
\renewcommand{\P}{\mathcal{P}}

\newcommand\st{:}
\newcommand{\setcond}[2]{\left\{ #1 \,:\, #2 \right\}}

\DeclareMathOperator    \relint         {rel\,int}
\DeclareMathOperator    \verts          {vert}
\DeclareMathOperator    \intr                   {int}
\DeclareMathOperator    \cl                     {cl}

\chardef\Myunderscore=`\_
\usepackage{hyperref}  
\newcommand\underscore{\Myunderscore\allowbreak}
\let\_=\underscore

\newcommand\githubsearchurl{https://github.com/mkoeppe/infinite-group-relaxation-code/search}
\usepackage{pgf}
\pgfkeyssetvalue{/sagefunc/bccz_counterexample}{\href{\githubsearchurl?q=\%22def+bccz_counterexample(\%22}{\sage{bccz\underscore{}counterexample}}}
\pgfkeyssetvalue{/sagefunc/bcdsp_arbitrary_slope}{\href{\githubsearchurl?q=\%22def+bcdsp_arbitrary_slope(\%22}{\sage{bcdsp\underscore{}arbitrary\underscore{}slope}}}
\pgfkeyssetvalue{/sagefunc/bhk_gmi_irrational}{\href{\githubsearchurl?q=\%22def+bhk_gmi_irrational(\%22}{\sage{bhk\underscore{}gmi\underscore{}irrational}}}
\pgfkeyssetvalue{/sagefunc/bhk_irrational}{\href{\githubsearchurl?q=\%22def+bhk_irrational(\%22}{\sage{bhk\underscore{}irrational}}}
\pgfkeyssetvalue{/sagefunc/bhk_slant_irrational}{\href{\githubsearchurl?q=\%22def+bhk_slant_irrational(\%22}{\sage{bhk\underscore{}slant\underscore{}irrational}}}
\pgfkeyssetvalue{/sagefunc/chen_4_slope}{\href{\githubsearchurl?q=\%22def+chen_4_slope(\%22}{\sage{chen\underscore{}4\underscore{}slope}}}
\pgfkeyssetvalue{/sagefunc/dg_2_step_mir}{\href{\githubsearchurl?q=\%22def+dg_2_step_mir(\%22}{\sage{dg\underscore{}2\underscore{}step\underscore{}mir}}}
\pgfkeyssetvalue{/sagefunc/dg_2_step_mir_limit}{\href{\githubsearchurl?q=\%22def+dg_2_step_mir_limit(\%22}{\sage{dg\underscore{}2\underscore{}step\underscore{}mir\underscore{}limit}}}
\pgfkeyssetvalue{/sagefunc/drlm_2_slope_limit}{\href{\githubsearchurl?q=\%22def+drlm_2_slope_limit(\%22}{\sage{drlm\underscore{}2\underscore{}slope\underscore{}limit}}}
\pgfkeyssetvalue{/sagefunc/drlm_3_slope_limit}{\href{\githubsearchurl?q=\%22def+drlm_3_slope_limit(\%22}{\sage{drlm\underscore{}3\underscore{}slope\underscore{}limit}}}
\pgfkeyssetvalue{/sagefunc/drlm_backward_3_slope}{\href{\githubsearchurl?q=\%22def+drlm_backward_3_slope(\%22}{\sage{drlm\underscore{}backward\underscore{}3\underscore{}slope}}}
\pgfkeyssetvalue{/sagefunc/gj_2_slope}{\href{\githubsearchurl?q=\%22def+gj_2_slope(\%22}{\sage{gj\underscore{}2\underscore{}slope}}}
\pgfkeyssetvalue{/sagefunc/gj_2_slope_repeat}{\href{\githubsearchurl?q=\%22def+gj_2_slope_repeat(\%22}{\sage{gj\underscore{}2\underscore{}slope\underscore{}repeat}}}
\pgfkeyssetvalue{/sagefunc/gj_forward_3_slope}{\href{\githubsearchurl?q=\%22def+gj_forward_3_slope(\%22}{\sage{gj\underscore{}forward\underscore{}3\underscore{}slope}}}
\pgfkeyssetvalue{/sagefunc/gmic}{\href{\githubsearchurl?q=\%22def+gmic(\%22}{\sage{gmic}}}
\pgfkeyssetvalue{/sagefunc/hildebrand_2_sided_discont_1_slope_1}{\href{\githubsearchurl?q=\%22def+hildebrand_2_sided_discont_1_slope_1(\%22}{\sage{hildebrand\underscore{}2\underscore{}sided\underscore{}discont\underscore{}1\underscore{}slope\underscore{}1}}}
\pgfkeyssetvalue{/sagefunc/hildebrand_2_sided_discont_2_slope_1}{\href{\githubsearchurl?q=\%22def+hildebrand_2_sided_discont_2_slope_1(\%22}{\sage{hildebrand\underscore{}2\underscore{}sided\underscore{}discont\underscore{}2\underscore{}slope\underscore{}1}}}
\pgfkeyssetvalue{/sagefunc/hildebrand_5_slope_22_1}{\href{\githubsearchurl?q=\%22def+hildebrand_5_slope_22_1(\%22}{\sage{hildebrand\underscore{}5\underscore{}slope\underscore{}22\underscore{}1}}}
\pgfkeyssetvalue{/sagefunc/hildebrand_5_slope_24_1}{\href{\githubsearchurl?q=\%22def+hildebrand_5_slope_24_1(\%22}{\sage{hildebrand\underscore{}5\underscore{}slope\underscore{}24\underscore{}1}}}
\pgfkeyssetvalue{/sagefunc/hildebrand_5_slope_28_1}{\href{\githubsearchurl?q=\%22def+hildebrand_5_slope_28_1(\%22}{\sage{hildebrand\underscore{}5\underscore{}slope\underscore{}28\underscore{}1}}}
\pgfkeyssetvalue{/sagefunc/hildebrand_discont_3_slope_1}{\href{\githubsearchurl?q=\%22def+hildebrand_discont_3_slope_1(\%22}{\sage{hildebrand\underscore{}discont\underscore{}3\underscore{}slope\underscore{}1}}}
\pgfkeyssetvalue{/sagefunc/kf_n_step_mir}{\href{\githubsearchurl?q=\%22def+kf_n_step_mir(\%22}{\sage{kf\underscore{}n\underscore{}step\underscore{}mir}}}
\pgfkeyssetvalue{/sagefunc/kzh_10_slope_1}{\href{\githubsearchurl?q=\%22def+kzh_10_slope_1(\%22}{\sage{kzh\underscore{}10\underscore{}slope\underscore{}1}}}
\pgfkeyssetvalue{/sagefunc/kzh_28_slope_1}{\href{\githubsearchurl?q=\%22def+kzh_28_slope_1(\%22}{\sage{kzh\underscore{}28\underscore{}slope\underscore{}1}}}
\pgfkeyssetvalue{/sagefunc/kzh_28_slope_2}{\href{\githubsearchurl?q=\%22def+kzh_28_slope_2(\%22}{\sage{kzh\underscore{}28\underscore{}slope\underscore{}2}}}
\pgfkeyssetvalue{/sagefunc/kzh_3_slope_param_extreme_1}{\href{\githubsearchurl?q=\%22def+kzh_3_slope_param_extreme_1(\%22}{\sage{kzh\underscore{}3\underscore{}slope\underscore{}param\underscore{}extreme\underscore{}1}}}
\pgfkeyssetvalue{/sagefunc/kzh_3_slope_param_extreme_2}{\href{\githubsearchurl?q=\%22def+kzh_3_slope_param_extreme_2(\%22}{\sage{kzh\underscore{}3\underscore{}slope\underscore{}param\underscore{}extreme\underscore{}2}}}
\pgfkeyssetvalue{/sagefunc/kzh_4_slope_param_extreme_1}{\href{\githubsearchurl?q=\%22def+kzh_4_slope_param_extreme_1(\%22}{\sage{kzh\underscore{}4\underscore{}slope\underscore{}param\underscore{}extreme\underscore{}1}}}
\pgfkeyssetvalue{/sagefunc/kzh_5_slope_fulldim_1}{\href{\githubsearchurl?q=\%22def+kzh_5_slope_fulldim_1(\%22}{\sage{kzh\underscore{}5\underscore{}slope\underscore{}fulldim\underscore{}1}}}
\pgfkeyssetvalue{/sagefunc/kzh_5_slope_fulldim_2}{\href{\githubsearchurl?q=\%22def+kzh_5_slope_fulldim_2(\%22}{\sage{kzh\underscore{}5\underscore{}slope\underscore{}fulldim\underscore{}2}}}
\pgfkeyssetvalue{/sagefunc/kzh_5_slope_fulldim_3}{\href{\githubsearchurl?q=\%22def+kzh_5_slope_fulldim_3(\%22}{\sage{kzh\underscore{}5\underscore{}slope\underscore{}fulldim\underscore{}3}}}
\pgfkeyssetvalue{/sagefunc/kzh_5_slope_fulldim_4}{\href{\githubsearchurl?q=\%22def+kzh_5_slope_fulldim_4(\%22}{\sage{kzh\underscore{}5\underscore{}slope\underscore{}fulldim\underscore{}4}}}
\pgfkeyssetvalue{/sagefunc/kzh_5_slope_fulldim_5}{\href{\githubsearchurl?q=\%22def+kzh_5_slope_fulldim_5(\%22}{\sage{kzh\underscore{}5\underscore{}slope\underscore{}fulldim\underscore{}5}}}
\pgfkeyssetvalue{/sagefunc/kzh_5_slope_fulldim_covers_1}{\href{\githubsearchurl?q=\%22def+kzh_5_slope_fulldim_covers_1(\%22}{\sage{kzh\underscore{}5\underscore{}slope\underscore{}fulldim\underscore{}covers\underscore{}1}}}
\pgfkeyssetvalue{/sagefunc/kzh_5_slope_fulldim_covers_2}{\href{\githubsearchurl?q=\%22def+kzh_5_slope_fulldim_covers_2(\%22}{\sage{kzh\underscore{}5\underscore{}slope\underscore{}fulldim\underscore{}covers\underscore{}2}}}
\pgfkeyssetvalue{/sagefunc/kzh_5_slope_fulldim_covers_3}{\href{\githubsearchurl?q=\%22def+kzh_5_slope_fulldim_covers_3(\%22}{\sage{kzh\underscore{}5\underscore{}slope\underscore{}fulldim\underscore{}covers\underscore{}3}}}
\pgfkeyssetvalue{/sagefunc/kzh_5_slope_fulldim_covers_4}{\href{\githubsearchurl?q=\%22def+kzh_5_slope_fulldim_covers_4(\%22}{\sage{kzh\underscore{}5\underscore{}slope\underscore{}fulldim\underscore{}covers\underscore{}4}}}
\pgfkeyssetvalue{/sagefunc/kzh_5_slope_fulldim_covers_5}{\href{\githubsearchurl?q=\%22def+kzh_5_slope_fulldim_covers_5(\%22}{\sage{kzh\underscore{}5\underscore{}slope\underscore{}fulldim\underscore{}covers\underscore{}5}}}
\pgfkeyssetvalue{/sagefunc/kzh_5_slope_fulldim_covers_6}{\href{\githubsearchurl?q=\%22def+kzh_5_slope_fulldim_covers_6(\%22}{\sage{kzh\underscore{}5\underscore{}slope\underscore{}fulldim\underscore{}covers\underscore{}6}}}
\pgfkeyssetvalue{/sagefunc/kzh_5_slope_q22_f10_1}{\href{\githubsearchurl?q=\%22def+kzh_5_slope_q22_f10_1(\%22}{\sage{kzh\underscore{}5\underscore{}slope\underscore{}q22\underscore{}f10\underscore{}1}}}
\pgfkeyssetvalue{/sagefunc/kzh_5_slope_q22_f10_2}{\href{\githubsearchurl?q=\%22def+kzh_5_slope_q22_f10_2(\%22}{\sage{kzh\underscore{}5\underscore{}slope\underscore{}q22\underscore{}f10\underscore{}2}}}
\pgfkeyssetvalue{/sagefunc/kzh_5_slope_q22_f10_3}{\href{\githubsearchurl?q=\%22def+kzh_5_slope_q22_f10_3(\%22}{\sage{kzh\underscore{}5\underscore{}slope\underscore{}q22\underscore{}f10\underscore{}3}}}
\pgfkeyssetvalue{/sagefunc/kzh_5_slope_q22_f10_4}{\href{\githubsearchurl?q=\%22def+kzh_5_slope_q22_f10_4(\%22}{\sage{kzh\underscore{}5\underscore{}slope\underscore{}q22\underscore{}f10\underscore{}4}}}
\pgfkeyssetvalue{/sagefunc/kzh_5_slope_q22_f2_1}{\href{\githubsearchurl?q=\%22def+kzh_5_slope_q22_f2_1(\%22}{\sage{kzh\underscore{}5\underscore{}slope\underscore{}q22\underscore{}f2\underscore{}1}}}
\pgfkeyssetvalue{/sagefunc/kzh_6_slope_1}{\href{\githubsearchurl?q=\%22def+kzh_6_slope_1(\%22}{\sage{kzh\underscore{}6\underscore{}slope\underscore{}1}}}
\pgfkeyssetvalue{/sagefunc/kzh_6_slope_fulldim_covers_1}{\href{\githubsearchurl?q=\%22def+kzh_6_slope_fulldim_covers_1(\%22}{\sage{kzh\underscore{}6\underscore{}slope\underscore{}fulldim\underscore{}covers\underscore{}1}}}
\pgfkeyssetvalue{/sagefunc/kzh_6_slope_fulldim_covers_2}{\href{\githubsearchurl?q=\%22def+kzh_6_slope_fulldim_covers_2(\%22}{\sage{kzh\underscore{}6\underscore{}slope\underscore{}fulldim\underscore{}covers\underscore{}2}}}
\pgfkeyssetvalue{/sagefunc/kzh_6_slope_fulldim_covers_3}{\href{\githubsearchurl?q=\%22def+kzh_6_slope_fulldim_covers_3(\%22}{\sage{kzh\underscore{}6\underscore{}slope\underscore{}fulldim\underscore{}covers\underscore{}3}}}
\pgfkeyssetvalue{/sagefunc/kzh_6_slope_fulldim_covers_4}{\href{\githubsearchurl?q=\%22def+kzh_6_slope_fulldim_covers_4(\%22}{\sage{kzh\underscore{}6\underscore{}slope\underscore{}fulldim\underscore{}covers\underscore{}4}}}
\pgfkeyssetvalue{/sagefunc/kzh_6_slope_fulldim_covers_5}{\href{\githubsearchurl?q=\%22def+kzh_6_slope_fulldim_covers_5(\%22}{\sage{kzh\underscore{}6\underscore{}slope\underscore{}fulldim\underscore{}covers\underscore{}5}}}
\pgfkeyssetvalue{/sagefunc/kzh_7_slope_1}{\href{\githubsearchurl?q=\%22def+kzh_7_slope_1(\%22}{\sage{kzh\underscore{}7\underscore{}slope\underscore{}1}}}
\pgfkeyssetvalue{/sagefunc/kzh_7_slope_2}{\href{\githubsearchurl?q=\%22def+kzh_7_slope_2(\%22}{\sage{kzh\underscore{}7\underscore{}slope\underscore{}2}}}
\pgfkeyssetvalue{/sagefunc/kzh_7_slope_3}{\href{\githubsearchurl?q=\%22def+kzh_7_slope_3(\%22}{\sage{kzh\underscore{}7\underscore{}slope\underscore{}3}}}
\pgfkeyssetvalue{/sagefunc/kzh_7_slope_4}{\href{\githubsearchurl?q=\%22def+kzh_7_slope_4(\%22}{\sage{kzh\underscore{}7\underscore{}slope\underscore{}4}}}
\pgfkeyssetvalue{/sagefunc/ll_strong_fractional}{\href{\githubsearchurl?q=\%22def+ll_strong_fractional(\%22}{\sage{ll\underscore{}strong\underscore{}fractional}}}
\pgfkeyssetvalue{/sagefunc/mlr_cpl3_a_2_slope}{\href{\githubsearchurl?q=\%22def+mlr_cpl3_a_2_slope(\%22}{\sage{mlr\underscore{}cpl3\underscore{}a\underscore{}2\underscore{}slope}}}
\pgfkeyssetvalue{/sagefunc/mlr_cpl3_b_3_slope}{\href{\githubsearchurl?q=\%22def+mlr_cpl3_b_3_slope(\%22}{\sage{mlr\underscore{}cpl3\underscore{}b\underscore{}3\underscore{}slope}}}
\pgfkeyssetvalue{/sagefunc/mlr_cpl3_c_3_slope}{\href{\githubsearchurl?q=\%22def+mlr_cpl3_c_3_slope(\%22}{\sage{mlr\underscore{}cpl3\underscore{}c\underscore{}3\underscore{}slope}}}
\pgfkeyssetvalue{/sagefunc/mlr_cpl3_d_3_slope}{\href{\githubsearchurl?q=\%22def+mlr_cpl3_d_3_slope(\%22}{\sage{mlr\underscore{}cpl3\underscore{}d\underscore{}3\underscore{}slope}}}
\pgfkeyssetvalue{/sagefunc/mlr_cpl3_f_2_or_3_slope}{\href{\githubsearchurl?q=\%22def+mlr_cpl3_f_2_or_3_slope(\%22}{\sage{mlr\underscore{}cpl3\underscore{}f\underscore{}2\underscore{}or\underscore{}3\underscore{}slope}}}
\pgfkeyssetvalue{/sagefunc/mlr_cpl3_g_3_slope}{\href{\githubsearchurl?q=\%22def+mlr_cpl3_g_3_slope(\%22}{\sage{mlr\underscore{}cpl3\underscore{}g\underscore{}3\underscore{}slope}}}
\pgfkeyssetvalue{/sagefunc/mlr_cpl3_h_2_slope}{\href{\githubsearchurl?q=\%22def+mlr_cpl3_h_2_slope(\%22}{\sage{mlr\underscore{}cpl3\underscore{}h\underscore{}2\underscore{}slope}}}
\pgfkeyssetvalue{/sagefunc/mlr_cpl3_k_2_slope}{\href{\githubsearchurl?q=\%22def+mlr_cpl3_k_2_slope(\%22}{\sage{mlr\underscore{}cpl3\underscore{}k\underscore{}2\underscore{}slope}}}
\pgfkeyssetvalue{/sagefunc/mlr_cpl3_l_2_slope}{\href{\githubsearchurl?q=\%22def+mlr_cpl3_l_2_slope(\%22}{\sage{mlr\underscore{}cpl3\underscore{}l\underscore{}2\underscore{}slope}}}
\pgfkeyssetvalue{/sagefunc/mlr_cpl3_n_3_slope}{\href{\githubsearchurl?q=\%22def+mlr_cpl3_n_3_slope(\%22}{\sage{mlr\underscore{}cpl3\underscore{}n\underscore{}3\underscore{}slope}}}
\pgfkeyssetvalue{/sagefunc/mlr_cpl3_o_2_slope}{\href{\githubsearchurl?q=\%22def+mlr_cpl3_o_2_slope(\%22}{\sage{mlr\underscore{}cpl3\underscore{}o\underscore{}2\underscore{}slope}}}
\pgfkeyssetvalue{/sagefunc/mlr_cpl3_p_2_slope}{\href{\githubsearchurl?q=\%22def+mlr_cpl3_p_2_slope(\%22}{\sage{mlr\underscore{}cpl3\underscore{}p\underscore{}2\underscore{}slope}}}
\pgfkeyssetvalue{/sagefunc/mlr_cpl3_q_2_slope}{\href{\githubsearchurl?q=\%22def+mlr_cpl3_q_2_slope(\%22}{\sage{mlr\underscore{}cpl3\underscore{}q\underscore{}2\underscore{}slope}}}
\pgfkeyssetvalue{/sagefunc/mlr_cpl3_r_2_slope}{\href{\githubsearchurl?q=\%22def+mlr_cpl3_r_2_slope(\%22}{\sage{mlr\underscore{}cpl3\underscore{}r\underscore{}2\underscore{}slope}}}
\pgfkeyssetvalue{/sagefunc/rlm_dpl1_extreme_3a}{\href{\githubsearchurl?q=\%22def+rlm_dpl1_extreme_3a(\%22}{\sage{rlm\underscore{}dpl1\underscore{}extreme\underscore{}3a}}}
\pgfkeyssetvalue{/sagefunc/automorphism}{\href{\githubsearchurl?q=\%22def+automorphism(\%22}{\sage{automorphism}}}
\pgfkeyssetvalue{/sagefunc/multiplicative_homomorphism}{\href{\githubsearchurl?q=\%22def+multiplicative_homomorphism(\%22}{\sage{multiplicative\underscore{}homomorphism}}}
\pgfkeyssetvalue{/sagefunc/projected_sequential_merge}{\href{\githubsearchurl?q=\%22def+projected_sequential_merge(\%22}{\sage{projected\underscore{}sequential\underscore{}merge}}}
\pgfkeyssetvalue{/sagefunc/restrict_to_finite_group}{\href{\githubsearchurl?q=\%22def+restrict_to_finite_group(\%22}{\sage{restrict\underscore{}to\underscore{}finite\underscore{}group}}}
\pgfkeyssetvalue{/sagefunc/interpolate_to_infinite_group}{\href{\githubsearchurl?q=\%22def+interpolate_to_infinite_group(\%22}{\sage{interpolate\underscore{}to\underscore{}infinite\underscore{}group}}}
\pgfkeyssetvalue{/sagefunc/two_slope_fill_in}{\href{\githubsearchurl?q=\%22def+two_slope_fill_in(\%22}{\sage{two\underscore{}slope\underscore{}fill\underscore{}in}}}
\pgfkeyssetvalue{/sagefunc/generate_example_e_for_psi_n}{\href{\githubsearchurl?q=\%22def+generate_example_e_for_psi_n(\%22}{\sage{generate\underscore{}example\underscore{}e\underscore{}for\underscore{}psi\underscore{}n}}}
\pgfkeyssetvalue{/sagefunc/chen_3_slope_not_extreme}{\href{\githubsearchurl?q=\%22def+chen_3_slope_not_extreme(\%22}{\sage{chen\underscore{}3\underscore{}slope\underscore{}not\underscore{}extreme}}}
\pgfkeyssetvalue{/sagefunc/psi_n_in_bccz_counterexample_construction}{\href{\githubsearchurl?q=\%22def+psi_n_in_bccz_counterexample_construction(\%22}{\sage{psi\underscore{}n\underscore{}in\underscore{}bccz\underscore{}counterexample\underscore{}construction}}}
\pgfkeyssetvalue{/sagefunc/gomory_fractional}{\href{\githubsearchurl?q=\%22def+gomory_fractional(\%22}{\sage{gomory\underscore{}fractional}}}
\pgfkeyssetvalue{/sagefunc/not_minimal_2}{\href{\githubsearchurl?q=\%22def+not_minimal_2(\%22}{\sage{not\underscore{}minimal\underscore{}2}}}
\pgfkeyssetvalue{/sagefunc/not_extreme_1}{\href{\githubsearchurl?q=\%22def+not_extreme_1(\%22}{\sage{not\underscore{}extreme\underscore{}1}}}
\pgfkeyssetvalue{/sagefunc/kzh_2q_example_1}{\href{\githubsearchurl?q=\%22def+kzh_2q_example_1(\%22}{\sage{kzh\underscore{}2q\underscore{}example\underscore{}1}}}
\pgfkeyssetvalue{/sagefunc/zhou_two_sided_discontinuous_cannot_assume_any_continuity}{\href{\githubsearchurl?q=\%22def+zhou_two_sided_discontinuous_cannot_assume_any_continuity(\%22}{\sage{zhou\underscore{}two\underscore{}sided\underscore{}discontinuous\underscore{}cannot\underscore{}assume\underscore{}any\underscore{}continuity}}}
\pgfkeyssetvalue{/sagefunc/extremality_test}{\href{\githubsearchurl?q=\%22def+extremality_test(\%22}{\sage{extremality\underscore{}test}}}
\pgfkeyssetvalue{/sagefunc/plot_2d_diagram}{\href{\githubsearchurl?q=\%22def+plot_2d_diagram(\%22}{\sage{plot\underscore{}2d\underscore{}diagram}}}
\pgfkeyssetvalue{/sagefunc/nice_field_values}{\href{\githubsearchurl?q=\%22def+nice_field_values(\%22}{\sage{nice\underscore{}field\underscore{}values}}}
\pgfkeyssetvalue{/sagefunc/ParametricRealFieldElement}{\href{\githubsearchurl?q=\%22def+ParametricRealFieldElement(\%22}{\sage{ParametricRealFieldElement}}}
\pgfkeyssetvalue{/sagefunc/ParametricRealField}{\href{\githubsearchurl?q=\%22def+ParametricRealField(\%22}{\sage{ParametricRealField}}}
\pgfkeyssetvalue{/sagefunc/kzh_minimal_has_only_crazy_perturbation_1}{\href{\githubsearchurl?q=\%22def+kzh_minimal_has_only_crazy_perturbation_1(\%22}{\sage{kzh\underscore{}minimal\underscore{}has\underscore{}only\underscore{}crazy\underscore{}perturbation\underscore{}1}}}

\DeclareRobustCommand\sage[1]{\texttt{#1}}
\DeclareRobustCommand\sagefunc[1]{\pgfkeys{/sagefunc/#1}}

\makeatletter 
\def\@fnsymbol#1{\ensuremath{\ifcase#1\or\star\or{\star\star}\or
   {\star{\star}\star}\or \dagger\or \ddagger\or
   \mathchar "278\or \mathchar "27B\or \|\or **\or \dagger\dagger
   \or \ddagger\ddagger \else\@ctrerr\fi}}
\makeatother

\begin{document}


\title{Equivariant Perturbation in \\Gomory and Johnson's Infinite Group 
  Problem.\\[1ex] V. Software for the continuous and discontinuous 1-row case}

\author{
  \name{%
    Chun Yu Hong\textsuperscript{a}$^{\ast}$\thanks{$^\ast$ The first author's contribution was done
      during a Research Experience for Undergraduates at the University of
      California, Davis.},
    Matthias K\"oppe\textsuperscript{b}$^{\ast\ast}$\thanks{$^{\ast\ast}$Corresponding author. Email:
      mkoeppe@math.ucdavis.edu} 
    and Yuan Zhou\textsuperscript{b}
  }
  \affil{%
    \textsuperscript{a}University of California, Berkeley, Department of Statistics, USA\\ 
    \textsuperscript{b}University of California, Davis, Department of Mathematics, USA}
}

\maketitle

\begin{abstract}
  We present software for investigations with cut-generating functions in the
  Gomory--Johnson model and extensions, implemented in the computer algebra system
  SageMath.
\end{abstract}


\begin{classcode}
  90C10; 90C11
\end{classcode}

\section{Introduction}

Consider the following question from the theory of linear inequalities over the
reals: Given a (finite) system $Ax \leq b$, exactly which linear inequalities
$\langle a,x\rangle \leq \beta$ are \emph{valid}, i.e., satisfied for every
$x$ that satisfies the given system?  The answer is given, of course, by the
Farkas Lemma, or, equivalently, by the strong duality theory of linear
optimization.  As is well-known, this duality theory is symmetric: The dual of
a linear optimization problem is again a linear optimization problem, and the
dual of the dual is the original (primal) optimization problem.

The question becomes much harder when all or some of the variables are
constrained to be integers.  The theory of valid linear inequalities here is
called \emph{cutting plane theory}.  Over the past 60 years, a vast body of
research has been carried out on this topic, the largest part of it regarding
the polyhedral combinatorics of integer hulls of particular families of
problems.  The general theory again is equivalent to the duality theory of
integer linear optimization problems.  Here the dual objects are not linear,
but \emph{superadditive} (or subadditive) functionals, making the general form of
this theory infinite-dimensional even though the original problem started out
with only finitely many variables.

These superadditive (or subadditive) functionals appear in integer linear
optimization in various concrete forms, for example in the form of
\emph{dual-feasible functions}
\cite{alves-clautiaux-valerio-rietz-2016:dual-feasible-book}, 
\emph{superadditive lifting functions} 
\cite{louveaux-wolsey-2003:lifting-superadditivity-mir-survey}, 
and 
\emph{cut-generating functions}~\cite{conforti2013cut}. 
\medbreak

In the present paper, we describe some aspects of our software
\cite{infinite-group-relaxation-code} for cut-generating functions in the
classic 1-row Gomory--Johnson \cite{infinite,infinite2} model. The model has a
single parameter, a number $f\in \R\setminus\Z$.    On the
primal side of the model, one has finite-support functions $y\colon \R\to
\Z_+$ that satisfy the group equation
\begin{equation}
  \label{eq:GP}
  \sum_{r \in \R} r\, y(r) \in f + \Z.
\end{equation}
We omit any further discussion of the primal side and how it arises as an
infinite-dimensional relaxation of integer linear optimization problems; we
refer the reader to the recent survey \cite{igp_survey,igp_survey_part_2} for
a complete exposition.
Our software only works on the dual side, where the main objects of interest
are the so-called \emph{minimal valid functions}.  By a 
characterization by Gomory--Johnson \cite{infinite} (see also \cite[Theorem
2.6]{igp_survey}), 
they are the $\Z$-periodic functions $\pi\colon \R\to\R$ 
that satisfy the following conditions:
\begin{subequations}\label{eq:minimal}
  \begin{align}
    &\pi(0)=0, \ \pi(f)=1 \label{eq:minimal:01}\\
    &\pi(x) \geq 0 &&\text{for } x \in \R \label{eq:minimal:nonneg}\\
    &\pi(x) + \pi(f-x) = 1 && \text{for } x\in \R 
                   && \text{(symmetry condition)} \label{eq:minimal:symm}\\
    &\pi(x) + \pi(y) \geq \pi(x+y) &&  \text{for } x,y \in \R 
                   && \text{(subadditivity)}.\label{eq:minimal:subadd}
  \end{align}
\end{subequations}

In the same way that finite-dimensional cutting plane theory has focused on
finding families of facet-defining valid inequalities, a large part of the
literature on the Gomory--Johnson model has focused on \emph{extreme
  functions}.  These are the minimal valid functions that are not a proper
convex combination of other minimal valid
functions. \autoref{fig:cont_and_discont_pwl_functions}--left shows the famous
extreme function \sagefunc{gmic}, the cut-generating function of the Gomory
mixed integer cut.  It is convenient to rephrase the extremality condition in
terms of \emph{perturbation functions}.  We say that a
function~$\tilde\pi\colon \R\to\R$ is an \emph{effective perturbation
  function} for the minimal valid function~$\pi$, if there exists $\epsilon>0$
such that $\pi\pm\epsilon\bar\pi$ are minimal valid functions.  The effective
perturbation functions form a vector space, denoted by
$\tilde\Pi^{\pi}(\R,\Z)$; thus $\pi$ is extreme if and only if this space is
trivial.


\medbreak

This paper is part of a series, dedicated to making the theory of
cut-generating functions in the Gomory--Johnson model algorithmic.  For the
1-row case, the focus lies on $\Z$-periodic piecewise linear functions~$\pi$,
which have an obvious finite representation in the computer.
It is an easy observation that the conditions~\eqref{eq:minimal} for
minimality are finitely checkable for such functions:  Consider a $\Z^2$-periodic
polyhedral complex $\Delta\P$ on $\R^2$ such that the \emph{subadditivity
  slack} $\Delta\pi(x,y) = \pi(x)+\pi(y)-\pi(x+y)$ is a linear function on the
relative interior of each face of~$\Delta\P$.  Then, in the continuous case, it suffices to check
subadditivity (i.e., nonnegativity of $\Delta\pi$) on the vertices ($0$-dimensional
faces) of this complex. In the discontinuous case, 
one additionally has to
consider finitely many directional limits to the vertices.

In contrast to minimality, testing extremality is much more subtle.  Making it
algorithmic for various classes of functions is the main technical
contribution of the present series of papers.  In this project, the study of spaces of
perturbation functions takes a central r\^ole.  

A proof of extremality from the published literature follows a standard
pattern, which can be phrased in terms of perturbation functions as follows.
Take an effective perturbation function $\tilde\pi$; the goal is to show that
$\tilde\pi = 0$.  Whenever $\Delta\pi(x,y) = 0$ (additivity) holds for some
$x, y$, the same is true for $\tilde\pi$.  By applying reasoning from
functional equations, in particular variants of the Gomory--Johnson Interval
Lemma, to these additivity relations, one infers that $\tilde\pi$ is linear on
certain intervals.  If the intervals of linearity, which we refer to as
\emph{covered intervals}, cover the domain of $\tilde\pi$, then $\tilde\pi$
belongs to a finite-dimensional space of piecewise linear perturbation
functions.  In this case, a matrix computation suffices to decide whether
$\tilde\pi = 0$.  However, it remained unclear whether these tools and this
proof pattern were sufficient for deciding extremality in all cases.

In the first paper in the present
series, Basu et al.  \cite{basu-hildebrand-koeppe:equivariant}, focusing on
the 1-row Gomory--Johnson model,  
developed additional tools for piecewise linear functions with rational
breakpoints 
and showed that a finite number of applications of these tools suffice for
deciding extremality in the case of piecewise linear
(possibly discontinuous) minimal valid functions with rational breakpoints.
Thus, Basu et al. \cite{basu-hildebrand-koeppe:equivariant} obtained the first
algorithm for testing extremality of a minimal valid function within this
class.

This algorithm from \cite{basu-hildebrand-koeppe:equivariant} works using the
grid $\frac1q\Z$, where $q$ is the least common multiple of all breakpoint
denominators.  In contrast, our software implements a \emph{grid-free variant}%
.  In the present paper, we provide a number of new results that underly
the grid-free extremality test; they are generalizations 
of results from the literature.
A practical benefit of our grid-free variant is that its empirical running
time does not strongly depend on $q$.  Moreover, it paves the way to
computations with functions that have irrational breakpoints; this topic is further
developed in \cite{koeppe-zhou:crazy-perturbation,koeppe-zhou:algo-paper}. It
is also the basis for computational investigations with parametric families of
extreme functions in \cite{koeppe-zhou:param-abstract-isco}.

\medbreak

Our software is a tool that enables mathematical exploration and research in
the domain of cut-generating functions.
It can also be used in an educational setting, where it enables
hands-on teaching about modern cutting plane theory based on cut-generating
functions.  It removes the limitations of hand-written proofs, which would be
dominated by tedious case analysis.

Our software is written in Python, making
use of the convenient and popular framework of the open-source computer
algebra system SageMath.  We begin our article with a brief overview about
this system (\autoref{s:sage}).

Next, in \autoref{s:pwl-functions}, we describe the anatomy of our main
objects, the $\Z$-periodic, possibly discontinuous piecewise linear
functions~$\pi\colon \R\to\R$, connecting the notation from the literature to
relevant Python functions and methods.  We also introduce the electronic
compendium \cite{zhou:extreme-notes} of extreme functions found in the
literature.  

The algorithmic minimality and extremality tests make use of certain
2-dimensional diagrams, associated with the polyhedral complex~$\Delta\P$,
which record the subadditivity and additivity properties of a given
function~$\pi$.  In \autoref{s:delta-p} we give a definition of this complex
and introduce these diagrams.
Then, in \autoref{s:minimality}, we describe the details of the automatic
minimality test.  In \autoref{sec:connected-covered-components} we explain the
computation of the covered intervals for a given minimal valid function.  This
is one of the ingredients of the extremality test, which we explain
in~\autoref{s:extremality}. 

In sections \ref{sec:procedrue} and~\ref{sec:finite-group}, we describe
transformations that can be applied to extreme functions and functionality of
our software for the closely related finite group relaxation model.  We end
our article with some comments regarding the documentation and test suite of
our software (\autoref{s:doc-test}).

\medbreak

The first version of our software \cite{infinite-group-relaxation-code} was
written by the first author, C.~Y. Hong, during a Research Experience for
Undergraduates 
in summer 2013.  It was
later revised and extended by 
M.~K\"oppe 
and again by 
Y.~Zhou.  The latter added the electronic compendium
and 
code that handles the case of
discontinuous functions.  Version 0.9 of our software was released in 2014 to
accompany the survey \cite{igp_survey,igp_survey_part_2}; the software has
received continuous updates by the second and third authors since. Two
further undergraduate students contributed to our software.  P.~Xiao
contributed some documentation and tests.  M.~Sugiyama contributed additional
functions to the compendium, and added code for 
superadditive lifting functions.

\section{About SageMath as a research and education platform in optimization}
\label{s:sage}

Our software is written in Python, making use of the convenient framework of
the open-source computer algebra system SageMath \cite{sage}.  It can be run
on a local installation of SageMath, or online via \emph{SageMathCloud}.

We briefly explain what sets SageMath apart from other computer algebra
systems, making it suitable as a research platform in optimization.  First, its
surface language is the popular programming language Python, altered only in a
minimal way by the SageMath pre-parser to provide some syntactic sugar for
mathematical notation.  This in contrast to systems such as Maple and
Mathematica, which have their own idiosyncratic surface languages.  Second,
while SageMath itself is written in Python and Cython, it interfaces to a
large number of open-source and commercial software packages, and as a
principle delegates all computations to a state-of-the-art library when
possible, rather than using its own implementation.  Relevant features for
research and education in optimization are the following:
\begin{itemize}
\item the interfaces to major numerical mixed-integer linear optimization
  solvers, CPLEX, Gurobi, COIN-OR CBC, in addition to GLPK, which is used as
  the default solver, as well as an interface to the convex optimization
  solver CVXOPT;
\item an interface to the exact (rational arithmetic) mixed
  integer linear optimization solver in the Parma Polyhedra Library;
\item a didactical implementation of the simplex method that is able to work
  over general ordered fields, with a didactical implementation of
  cutting-plane methods \cite[\#18805]{sage-trac} in development\footnote{In
    the development of SageMath, each code addition and change is subject to
    (non-blind) peer review, using the ticket system Trac \cite{sage-trac}.};
\item a textbook view on the numerical solvers \cite[\#18804]{sage-trac} in
  development;  
\item interfaces to state-of-the-art polyhedral computation software,
  including the Parma Polyhedra Library, cddlib, TOPCOM, Normaliz, and
  polymake, as well as an implementation of
  polyhedral computation methods for general ordered fields;
\item interfaces to state-of-the-art lattice-point computation software, LattE
  integrale, 4ti2, and Normaliz, as well as its own implementation of
  lattice-point enumeration.
\end{itemize}

\section{Continuous and discontinuous piecewise linear $\Z$-periodic functions}
\label{s:pwl-functions}

The main objects of our code are the $\Z$-periodic functions~$\pi\colon \R\to\R$.  Our code
is limited to the case of piecewise linear functions, which are allowed to be
discontinuous; see the definition below.  
In our code, the periodicity of the functions is implicit; the functions
are represented by their restriction to the interval $[0,1]$.\footnote{The
  functions are instances of the class \sage{FastPiecewise}, which extends an existing
  SageMath class for piecewise linear functions.}
They can be constructed in various ways using Python functions named
\sage{piecewise\_function\_from\_breakpoints\_and\_values} etc.

Our software includes an electronic compendium, which is up-to-date with our
knowledge on extreme functions. It is accessible by the Python module named
\sage{extreme\_functions}.  
One can use the help system of Sage, by typing the name of an extreme function 
followed by a question mark, to access the documentation string of the
function, which provides a discussion of parameters, bibliographic information,
etc.; see \autoref{tab:pwl-funtions}. 
\autoref{tab:gmic-in-compendium} shows the source code of \sage{gmic} in the
electronic compendium.  It constructs the piecewise linear extreme function from the
breakpoints and the values at the breakpoints.  This function is the
cut-generating function of the famous Gomory mixed-integer cut
\cite{gomory1960algorithm}. 

The source code of the electronic compendium provides many other examples of
constructions of functions. An early version of the electronic compendium has been described in 
the paper \cite{zhou:extreme-notes}. The reader may also refer to the survey
\cite[Tables 1--4]{igp_survey}, which 
shows the graphs of these functions for default parameters next to their
names.  Since the publication of \cite{zhou:extreme-notes,igp_survey}, additional extreme
functions have been implemented in the compendium. They include the family of extreme
functions with an arbitrary number of slopes
(\sage{extreme\_functions.bcdsp\_arbitrary\_slope}) that was recently
constructed by Basu--Conforti--Di Summa--Paat  \cite{bcdsp:arbitrary-slopes},
the family of $\text{CPL}_3^=$ functions (\sage{mlr\_cpl3\_...}) that was
obtained by Miller--Li--Richard \cite{Miller-Li-Richard2008} from an
approximate lifting of  superadditive functions,
and new parametric families and sporadic extreme functions (\sage{kzh\_...})
that were found by K\"oppe--Zhou
\cite{koeppe-zhou:extreme-search, koeppe-zhou:param-abstract-isco} using
computer-based search. 

\begin{table}
  \caption{The construction of the \sage{gmic} function in the compendium.}
  \label{tab:gmic-in-compendium}
  \tiny 
  \begin{tabular}{@{}p{\linewidth}@{}}
    \toprule
	\begin{verbatim}
	def gmic(f=4/5, field=None, conditioncheck=True):
    	"(docstring elided to save space)"
    	if conditioncheck and not bool(0 < f < 1):
        		raise ValueError, "Bad parameters. Unable to construct the function."
    	gmi_bkpt = [0,f,1]
    	gmi_values = [0,1,0]
    	return piecewise_function_from_breakpoints_and_values(gmi_bkpt, gmi_values, field=field)
	\end{verbatim}
    \\
    \bottomrule
  \end{tabular}
\end{table}

A piecewise linear function $\pi$ can be plotted using the standard SageMath function
\sage{plot($\pi$)}, or using our function
\sage{plot\_with\_colored\_slopes($\pi$)}, which assigns a different color to
each different slope value that a linear piece takes.\footnote{See also our
  function \sage{number\_of\_slopes}.  We refer the reader to \cite[section 2.4]{igp_survey} 
  for a discussion of the number of slopes of extreme functions, 
  and \cite{bcdsp:arbitrary-slopes} and \sagefunc{bcdsp_arbitrary_slope} for the
  latest developments in this direction.}

Random piecewise linear functions can be generated by calling the Python function named \sage{random\_piecewise\_function}. By specifying the parameters, one obtains random continuous or discontinuous piecewise linear functions with prescribed properties, which could be useful to experimentation and exploration. \autoref{fig:cont_and_discont_pwl_functions}--right shows a random discontinuous function generated by 

\begin{scriptsize}
  \begin{tabular}{@{}p{\linewidth}@{}}
	\begin{verbatim}
	h = random_piecewise_function(xgrid=5, ygrid=5, continuous_proba=1/3, symmetry=True).
	\end{verbatim}
  \end{tabular}
 \end{scriptsize}
This particular discontinuous function can also be set up by either of the following commands.

\begin{scriptsize}
  \begin{tabular}{@{}p{\linewidth}@{}}
	\begin{verbatim}
	h = piecewise_function_from_breakpoints_and_limits(bkpt=[0, 1/5, 2/5, 3/5, 4/5, 1],
	        limits=[(0, 0, 0), (1, 1, 1), (2/5, 2/5, 0), (1/2, 3/5, 2/5), (3/5, 1, 3/5), (0, 0, 0)])

	h = piecewise_function_from_breakpoints_and_limits(bkpt=[0, 1/5, 2/5, 3/5, 4/5, 1], 
	        limits=[{-1:0, 0:0, 1:0}, {-1:1, 0:1, 1:1}, {-1:0, 0:2/5, 1:2/5}, 
	                {-1:2/5, 0:1/2, 1:3/5}, {-1:3/5, 0:3/5, 1:1}, {-1:0, 0:0, 1:0}])
	\end{verbatim}
  \end{tabular}
 \end{scriptsize}
\begin{figure}
\centering
\begin{minipage}{.49\textwidth}
\centering
\includegraphics[width=.8\linewidth]{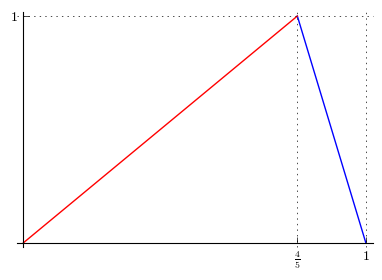}
\end{minipage}
\begin{minipage}{.49\textwidth}
\centering
\includegraphics[width=.8\linewidth]{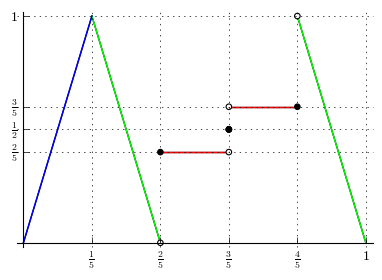}
\end{minipage}
\caption{Two piecewise linear functions, as plotted by the command
  \sage{plot\_with\_colored\_slopes(h)}. \textit{Left,} continuous extreme
  function \sage{h = gmic()}. \textit{Right,} random discontinuous function
  \sage{h = equiv5\_random\_discont\_1()}, generated by \sage{random\_piecewise\_function(xgrid=5, ygrid=5, continuous\_proba=1/3, symmetry=True)}.}
\label{fig:cont_and_discont_pwl_functions}
\end{figure}
\medbreak

In the following, we connect to the systematic notation introduced in
\cite[section 2.1]{basu-hildebrand-koeppe:equivariant}; see also 
\cite{igp_survey,igp_survey_part_2}.
We suppress the details of the internal representation of a $\Z$-periodic piecewise linear function $\pi$ in our code; instead we explain the
main ways in which the data of the function are accessed; see \autoref{tab:pwl-funtions} for a sample Sage session.

\begin{description}
\item[\sage{$\pi$.end\_points()}] is a list
  $0=x_0 < x_1 < \dots < x_{n-1} < x_n=1$ of possible
  breakpoints%
    \footnote{If the function~$\pi$ has been constructed with
      \sage{merge=True} (the default), then it is guaranteed that all end
      points $x_i$, with the possible exception of $0$ and $1$, are actual
      breakpoints of the $\Z$-periodic function~$\pi$.}
  of the function in $[0,1]$.  
  In the notation from \cite{basu-hildebrand-koeppe:equivariant,igp_survey,igp_survey_part_2}, 
  these endpoints are extended periodically as
  \begin{math}
    \B = \{\, x_0 + t, x_1 + t, \dots, x_{n-1}+t\st
    t\in\Z\,\}
  \end{math}.  Then the set of 0-dimensional faces is defined to be the collection of
  singletons, $\bigl\{\, \{ x \} \st x\in B\,\bigr\}$, 
  and the set of one-dimensional faces to be the collection of closed intervals,
  \begin{math}
    \bigl\{\, [x_i+t, x_{i+1}+t] \st i=0, \dots, {n-1} \text{ and } t\in\Z \,\bigr\}. 
  \end{math}
  Together, we obtain $\P = \P_{\B}$, a locally finite 
  polyhedral
  complex, 
  periodic modulo~$\Z$.
\item[\sage{$\pi$.values\_at\_end\_points()}] is a list 
  of the function values $\pi(x_i)$, 
  $i=0, \dots,n$. 
  This list is most useful for continuous piecewise linear functions, as 
  indicated by \sage{$\pi$.is\_continuous()}, in which case the function is
  defined on the intervals $[x_i, x_{i+1}]$ by linear interpolation.
\item[\sage{$\pi$.limits\_at\_end\_points()}] provides data for the
  general, possibly discontinuous case in the form of 
  a list \sage{limits} of 3-tuples, 
  with
  \begin{align*}
    \sage{limits[$i$][0]} &= \pi(x_i) \\
    \sage{limits[$i$][1]} &= \pi(x_i^+) = \lim_{x\to x_i, x>x_i} \pi(x)\\
    \sage{limits[$i$][-1]} &= \pi(x_i^-) = \lim_{x\to x_i, x<x_i} \pi(x). 
  \end{align*}
  The function is defined on the open intervals $(x_i, x_{i+1})$ by linear
  interpolation of the limit values $\pi(x_i^+)$, $\pi(x_{i+1}^-)$. 
\item[\sage{$\pi$($x$)} and \sage{$\pi$.limits($x$)}] evaluate the function at
  $x$ and provide the 3-tuple of its limits at $x$, respectively.
\item[\sage{$\pi$.which\_function($x$)}] returns a linear function, denoted
  $\pi_I\colon\R\to\R$ in \cite{basu-hildebrand-koeppe:equivariant,igp_survey,igp_survey_part_2},
  where $I$ is the smallest face of~$\P$ containing $x$, so $\pi(x) =
  \pi_I(x)$ for $x \in \relint(I)$. 
\end{description}

\begin{table}
  \caption{A sample Sage session, illustrating the basic use of a piecewise linear function and the help system.}
  \label{tab:pwl-funtions}
  \tiny 
  \begin{tabular}{@{}p{\linewidth}@{}}
    \toprule
	\begin{verbatim}
	## First load the code.
	sage: import igp; from igp import *
	INFO: 2016-08-08 16:49:21,594 Welcome to the infinite-group-relaxation-code. DON'T PANIC. See demo.sage for instructions.
	
	## The documentation string of each function reveals its optional arguments, usage examples, and bibliographic information. 
	sage: gmic?
	[...]
	Signature:      gmic(f=4/5, field=None, conditioncheck=True)
	Docstring:
	   Summary:
	      * Name: GMIC (Gomory mixed integer cut);
	
	      * Infinite (or Finite); Dim = 1; Slopes = 2; Continuous; Analysis of subadditive polytope method;
	
	      * Discovered [55] p.7-8, Eq.8;
	
	      * Proven extreme (for infinite group) [60] p.377, thm.3.3; (finite group) [57] p.514, Appendix 3.
	
	      * (Although only extremality has been established in literature, the same proof shows that) gmic is a facet.
	
	   Parameters:
	      f (real) in (0,1).
	
	   Examples:
	      [61] p.343, Fig. 1, Example 1
	
	         sage: logging.disable(logging.INFO)             # Suppress output in automatic tests.
	         sage: h = gmic(4/5)
	         sage: extremality_test(h, False)
	         True
	
	   Reference:
	      [55]: R.E. Gomory, An algorithm for the mixed integer problem, Tech. Report RM-2597, RAND Corporation, 1960.
	
	      [57]: R.E. Gomory, Some polyhedra related to combinatorial problems, Linear Algebra and its Application 2 (1969) 451-558.

	      [60]: R.E. Gomory and E.L. Johnson, Some continuous functions related to corner polyhedra, part II, 
	            Mathematical Programming 3 (1972) 359-389.
	
	      [61]: R.E. Gomory and E.L. Johnson, T-space and cutting planes, Mathematical Programming 96 (2003) 341-375.
	      
	## We load the GMIC function and store it in variable h.
	sage: h = gmic()
	INFO: 2016-08-08 16:51:31,048 Rational case.

	## We query the data of the GMIC function.
	sage: h.end_points()
	[0, 4/5, 1]
	sage: h.values_at_end_points()
	[0, 1, 0]
	sage: h(4/5)
	1
	sage: h.which_function(1/2)
	<FastLinearFunction 5/4*x>

	## Plot the function.
	sage: plot_with_colored_slopes(h)
	
	## Next, we construct a discontinuous piecewise linear function and query its data.
	sage: h = piecewise_function_from_breakpoints_and_limits(bkpt=[0, 1/5, 2/5, 3/5, 4/5, 1], 
		          limits=[(0, 0, 0), (1, 1, 1), (2/5, 2/5, 0), (1/2, 3/5, 2/5), (3/5, 1, 3/5), (0, 0, 0)])
	sage: h.limits_at_end_points()
	[[0, 0, 0],
	 [1, 1, 1],
	 [2/5, 2/5, 0],
	 [1/2, 3/5, 2/5],
	 [3/5, 1, 3/5],
	 [0, 0, 0]]
	\end{verbatim}
    \\
    \bottomrule
  \end{tabular}
\end{table}

\section{The diagrams of the decorated 2-dimensional polyhedral complex $\Delta\P$}
\label{s:delta-p}

We now describe certain 2-dimensional diagrams which record the
subadditivity and additivity properties of a given function.  
These diagrams, in the continuous case, have appeared extensively in
\cite{igp_survey,igp_survey_part_2,zhou:extreme-notes}.  An example for the
discontinuous case appeared in \cite{zhou:extreme-notes}.  
We have engineered these diagrams from earlier forms that can be found in
\cite{tspace} (for the discussion of the \sage{merit\_index}) and in
\cite{basu-hildebrand-koeppe:equivariant}, to become 
power tools for the modern cutgeneratingfunctionologist.
Not only is the minimality of a given function immediately apparent on the
diagram, but also the extremality proof for a given class of piecewise minimal
valid functions follows a standard pattern that draws from these diagrams.  See
\cite[prelude]{igp_survey_part_2} and \cite[sections 2 and
4]{zhou:extreme-notes} for examples of such proofs.

\subsection{The polyhedral complex and its faces}
Following
\cite{basu-hildebrand-koeppe:equivariant,igp_survey,igp_survey_part_2}, we
introduce the function 
\[\Delta\pi \colon \R \times \R \to \R,\quad \Delta\pi(x,y) =
  \pi(x)+\pi(y)-\pi(x+y),\] which measures the
slack in the subadditivity condition.\footnote{It is available in the code as
  \sage{delta\_pi($\pi$, $x$, $y$)}; in \cite{infinite}, it was called $\nabla(x,y)$.}  
Thus, if $\Delta\pi(x,y)<0$, subadditivity is violated at $(x, y)$; 
if $\Delta\pi(x,y)=0$, additivity holds at $(x,y)$; 
and if $\Delta\pi(x,y)>0$, we have strict subadditivity at $(x,y)$.
The piecewise linearity of $\pi(x)$ 
induces piecewise linearity of $\Delta\pi(x,y)$.  To express the domains of
linearity of $\Delta\pi(x,y)$, and thus domains of additivity and strict
subadditivity, we introduce the two-dimensional polyhedral complex
$\Delta\P$. 
The faces $F$ of the complex are defined as follows. Let $I, J, K \in
\P$, so each of $I, J, K$ is either a breakpoint of $\pi$ or a closed
interval delimited by two consecutive breakpoints. Then 
$$ F = F(I,J,K) = \setcond{\,(x,y) \in \R \times \R}{x \in I,\, y \in J,\, x + y \in
  K\,}.$$ 
In our code, a face is represented by an instance of the class \sage{Face}. 
It is constructed from $I, J, K$ and is represented by 
the list of vertices of $F$ and its projections $I'=p_1(F)$, $J'=p_2(F)$, $K'=
p_3(F)$,
where $p_1, p_2, p_3 \colon \R \times \R \to \R$ are defined as 
$p_1(x,y)=x$,  $p_2(x,y)=y$, $p_3(x,y) = x+y$.
The vertices $\verts(F)$ are obtained by first listing the basic solutions
$(x,y)$ where $x$, $y$, and $x+y$ are fixed to endpoints of $I$, $J$, and $K$,
respectively, and then filtering the feasible solutions. 
The three projections are then computed from the list of
vertices.%
  \footnote{We do not use the formulas for the projections given by
    \cite[Proposition 3.3]{bhk-IPCOext}, \cite[equation
    (3.11)]{igp_survey}.} 
Due to the $\Z$-periodicity of $\pi$, we can represent a face as a subset
of $[0,1]\times[0,1]$.  See \autoref{fig:construct_a_face} for an example. 
Because of the importance of the projection
$p_3(x,y)=x + y$, it is convenient to imagine a third, $(x+y)$-axis in
addition to the $x$-axis and the $y$-axis, which
traces the bottom border for $0 \leq x+y \leq 1$ and then the right border for
$1 \leq x+y \leq 2$.  To make room for this new axis, the $x$-axis should be
drawn on the top border of the diagram.
\begin{figure}[tp]
  \centering
  \includegraphics[width=.5\linewidth]{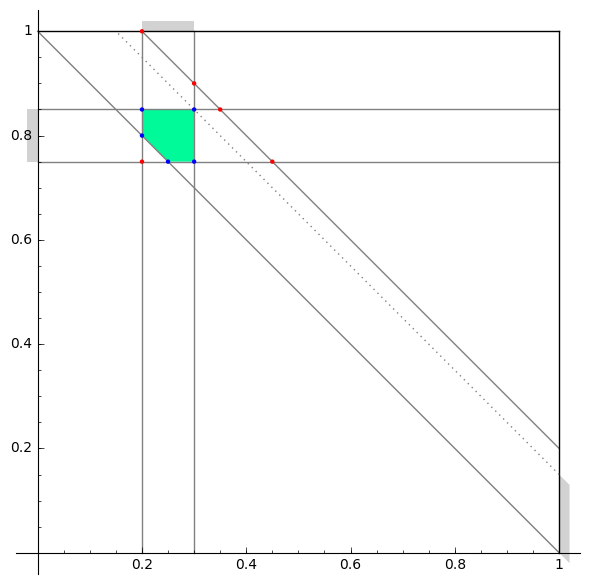}
  \caption{An example of a face $F = F(I, J, K)$ of the 2-dimensional
    polyhedral complex $\Delta\P$, set up by \sage{F = Face([[0.2, 0.3],
      [0.75, 0.85], [1, 1.2]])}.  It has vertices (\emph{blue})
    $(0.2, 0.85)$, $(0.3, 0.75)$, $(0.3, 0.85)$, $(0.2, 0.8)$, $(0.25, 0.75)$,
    whereas the other basic solutions (\emph{red})
    $(0.2, 0.75)$, $(0.2, 1)$, $(0.3, 0.9)$, $(0.35, 0.85)$, $(0.45, 0.75)$
    are filtered out because they are infeasible. 
    The face $F$ has projections (\emph{gray shadows})
    $I' = p_1(F) = [0.2, 0.3]$ (\emph{top border}), $J' = p_2(F) = [0.75,
    0.85]$ (\emph{left border}), and $K'
      = p_3(F) = [1, 1.15]$ (\emph{right border}). Note that $K'\subsetneq K$. 
  } 
  \label{fig:construct_a_face}
\end{figure}

\subsection{\sage{plot\_2d\_diagram\_with\_cones}}

We 
now 
explain the first version of the 
2-dimensional
diagrams, plotted by the function
\sage{plot\_2d\_diagram\_with\_cones($\pi$)}; see
\autoref{fig:2d_diagram_with_cones}
.
\begin{figure}[t]
\centering
\begin{minipage}{.49\textwidth}
\centering
\includegraphics[width=.9\linewidth]{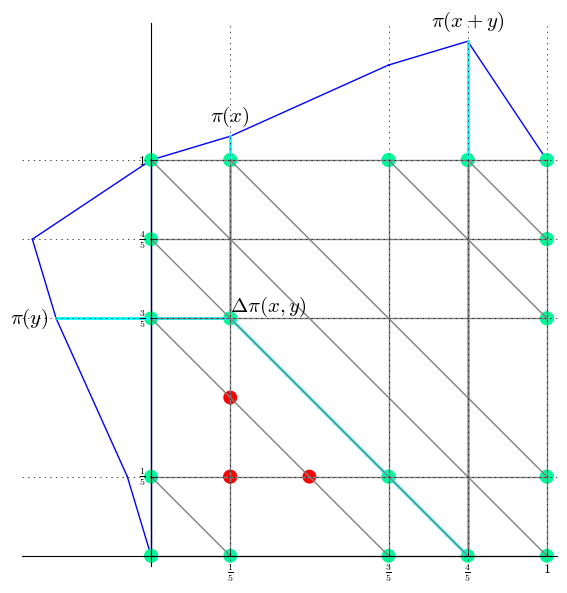}
\end{minipage}
\begin{minipage}{.49\textwidth}
\centering
\includegraphics[width=.9\linewidth]{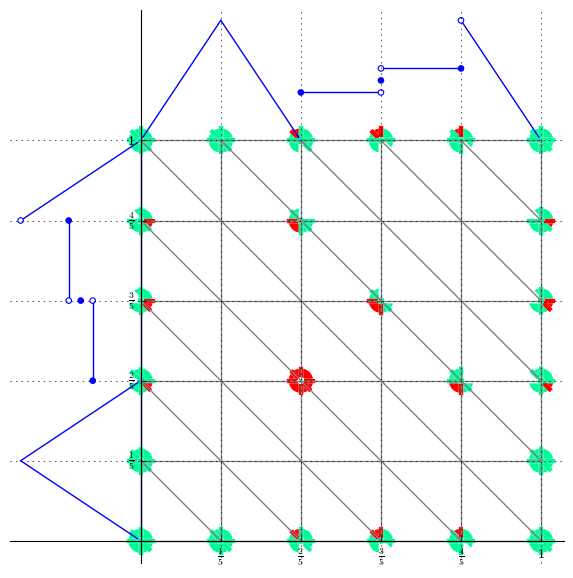}
\end{minipage}
\caption{Two diagrams of functions and their polyhedral complexes $\Delta\P$ with colored cones at $\verts(\Delta\P)$, as plotted by the command \sage{plot\_2d\_diagram\_with\_cones(h)}. \textit{Left}, continuous function \sage{h = not\_minimal\_2()}. \textit{Right}, random discontinuous function \sage{h = equiv5\_random\_discont\_1()}.}
\label{fig:2d_diagram_with_cones}
\end{figure}
At the border of these diagrams, the function $\pi$ is shown twice
(\emph{blue}), 
along the $x$-axis (\emph{top border}) and along the $y$-axis (\emph{left
  border}). 
The solid grid lines in the diagrams are determined by the breakpoints of
$\pi$: vertical, horizontal and diagonal grid lines correspond to values where
$x$, $y$ and $x+y$ are breakpoints of $\pi$, respectively. 
The vertices of the complex $\Delta\P$ are the intersections of these
grid lines. 

\textbf{In the continuous case}, 
we indicate the sign of $\Delta\pi(x,y)$ for all vertices by colored dots on
the diagram: \emph{red} indicates $\Delta\pi(x,y)<0$ (subadditivity is
violated); \emph{green} indicates $\Delta\pi(x,y)=0$ (additivity holds). 

\begin{example} 
  In \autoref{fig:2d_diagram_with_cones} (left), showing the 2-dimensional
  diagram of the function~$\pi = \sage{not\_minimal\_2()}$, the vertex
  $(x, y)=(\frac{1}{5}, \frac{3}{5})$ is marked green, since
\begin{align*}
\Delta\pi(\tfrac{1}{5},\tfrac{3}{5}) & =
                                       \pi(\tfrac{1}{5})+\pi(\tfrac{3}{5})-\pi(\tfrac{4}{5})
= \tfrac{1}{5} +\tfrac{4}{5} -1 = 0.
\end{align*}
\end{example} 

\textbf{In the discontinuous case}, beside the subadditivity slack $\Delta\pi(x,y)$ at a vertex $(x, y)$, one also needs to study the limit value of $\Delta\pi$ at the vertex $(x,y)$ approaching from the interior of a face $F \in \Delta\P$ containing the vertex $(x,y)$. This limit value is defined by
\[\Delta\pi_F(x,y) = \lim_{\substack{(u,v) \to (x,y)\\ (u,v) \in \relint(F)}}
  \Delta\pi(u,v), \quad \text{where } F \in \Delta\P \text{ such that
  } (x, y) \in F.\]  We indicate the sign of $\Delta\pi_F(x,y)$ by a colored
cone inside $F$ pointed at the vertex $(x, y)$ on the diagram. There could be
up to $12$ such cones (including rays for one-dimensional faces $F$) around a vertex $(x,
y)$.

\begin{example}
  In \autoref{fig:2d_diagram_with_cones} (right), showing the 2-dimensional
  diagram of the function $\pi = \sage{equiv5\_random\_discont\_1()}$, the lower
  right corner $(x, y)=(\frac{2}{5}, \frac{4}{5})$ of the face
  $F = F(I, J, K)$ with
  $I =[\frac{1}{5}, \frac{2}{5}]$, $J=[\frac{4}{5}, 1]$, $K=[1, \frac{6}{5}]$ is
  green, since
  \begin{align*}
    \Delta\pi_F(x,y) & = \lim_{\substack{(u,v) \to (\frac{2}{5},\frac{4}{5})\\ (u,v) \in \relint(F)}} \Delta\pi(u,v) \\
                     & = \lim_{u\to\frac{2}{5},\; u < \frac{2}{5}}\pi(u)+\lim_{v\to\frac{4}{5}, \; v > \frac{4}{5}}\pi(v)-\lim_{w\to\frac{6}{5},\; w <\frac{6}{5}}\pi(w) \\
                     & = \pi(\tfrac{2}{5}^-) +\pi(\tfrac{4}{5}^+) -\pi(\tfrac{1}{5}^-) \quad \text{(as } \pi(\tfrac{6}{5}^-) =\pi(\tfrac{1}{5}^-) \text{ by periodicity)} \\
                     & = 0 + 1 - 1 = 0.
  \end{align*}
  The horizontal ray to the left of the same vertex
  $(x, y)=(\frac{2}{5}, \frac{4}{5})$ is red, because approaching from the
  one-dimensional face $F' = F(I', J', K')$ that contains $(x,y)$, with
  $I' =[\frac{1}{5}, \frac{2}{5}]$, $J'=\{\frac{4}{5}\}$, $K'=[1, \frac{6}{5}]$,
  we have the limit value
  \[\Delta\pi_{F'}(x,y) = \lim_{\substack{(u,v) \to
        (\frac{2}{5},\frac{4}{5})\\ (u,v) \in \relint(F')}} \hspace{-1.5em}\Delta\pi(u,v) =
    \lim_{\substack{u\to\frac{2}{5} \\ u <
        \frac{2}{5}}}\pi(u)+\pi(\tfrac{4}{5})-\lim_{\substack{w\to\frac{6}{5}\\
        w <\frac{6}{5}}}\pi(w) = 0 + \tfrac{3}{5} - 1 < 0.\]
\end{example}

\subsection{\sage{plot\_2d\_diagram} and additive faces}

Now assume that $\pi$ is a subadditive function. Then there are no red dots or cones on the above diagram of the complex $\Delta\P$. 
See \autoref{fig:2d_diagrams_continuous_function} (left) and
\autoref{fig:2d_diagrams_discontinuous_function} (left) 
for illustrations of the continuous and discontinuous cases.

  \begin{figure}[tp]
    \centering
    \begin{minipage}{.49\textwidth}
      \centering
      \includegraphics[width=.9\linewidth]{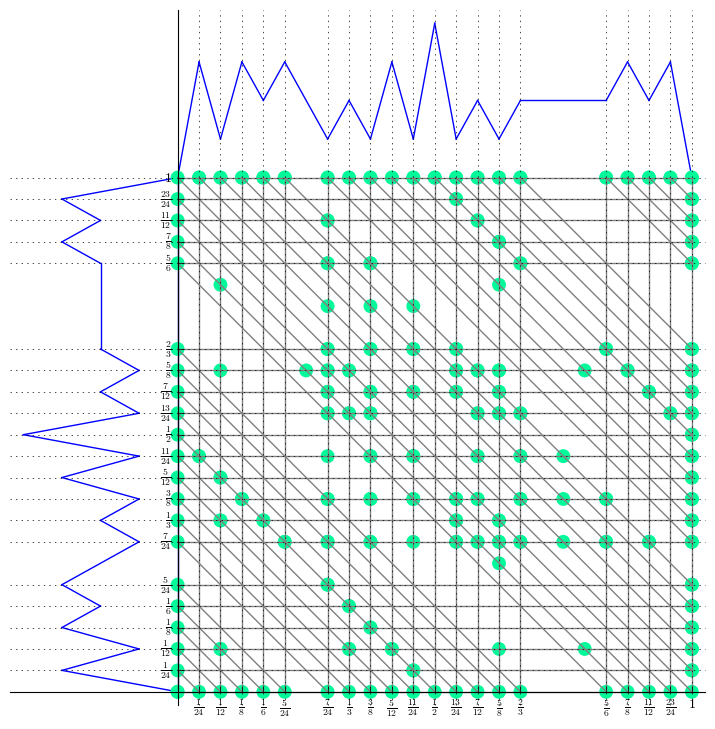}
    \end{minipage}
    \begin{minipage}{.49\textwidth}
      \centering
      \includegraphics[width=.9\linewidth]{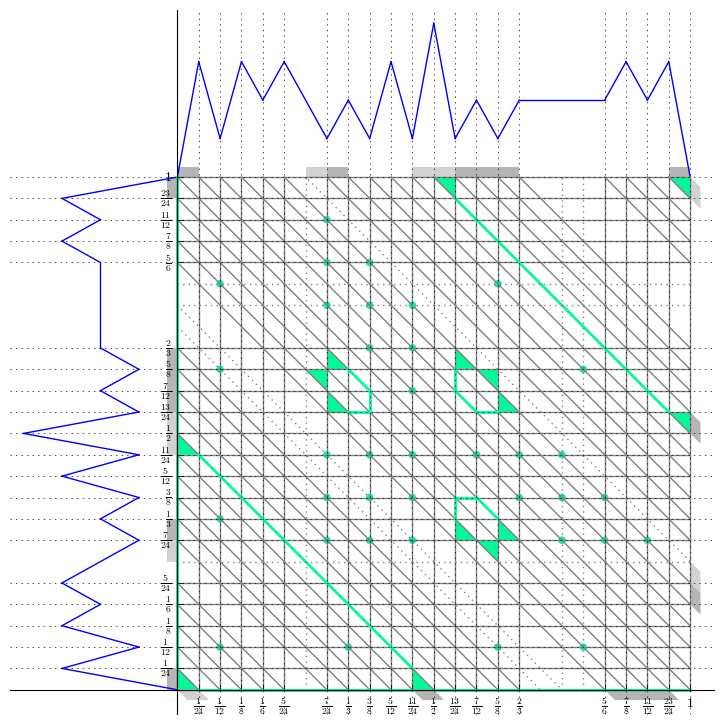}
    \end{minipage}
    \caption{Diagrams of $\Delta\P$ of a continuous function \sage{h
        = example7slopecoarse2()}, with (\textit{left}) additive vertices as
      plotted by the command \sage{plot\_2d\_diagram\_with\_cones(h)};
      (\textit{right}) maximal additive faces as plotted by the command
      \sage{plot\_2d\_diagram(h)}.}
    \label{fig:2d_diagrams_continuous_function}
  \end{figure}


\begin{figure}[tp]
\centering
\begin{minipage}{.49\textwidth}
\centering
\includegraphics[width=.9\linewidth]{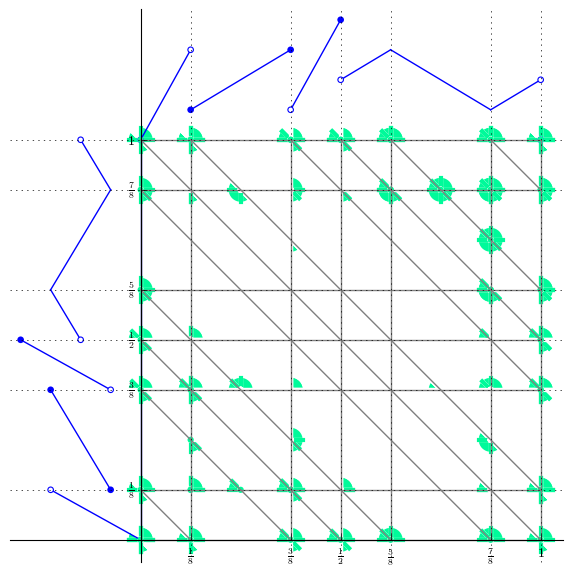}
\end{minipage}
\begin{minipage}{.49\textwidth}
\centering
\includegraphics[width=.9\linewidth]{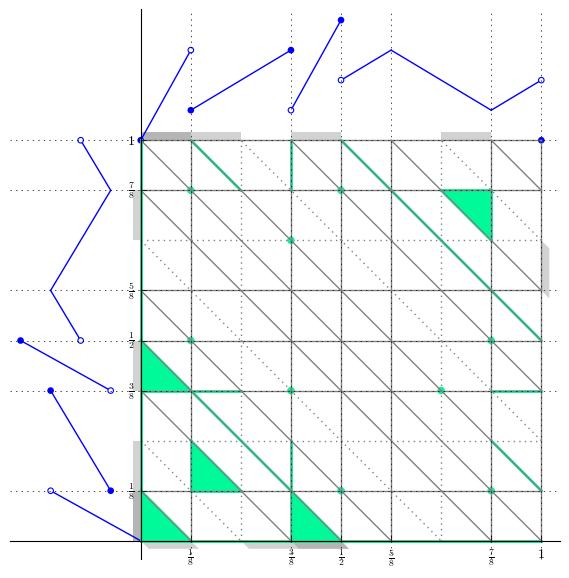}
\end{minipage}
\caption{Diagrams of $\Delta\P$ of a discontinuous function \sage{h = hildebrand\underscore discont\underscore 3\underscore slope\underscore 1()}, with (\textit{left}) additive limiting cones as plotted by the command \sage{plot\_2d\_diagram\_with\_cones(h)}; (\textit{right}) additive faces as plotted by the command \sage{plot\_2d\_diagram(h)}.}
\label{fig:2d_diagrams_discontinuous_function}
\end{figure}

\textbf{For a continuous subadditive function $\pi$}, we say that a face $F \in
\Delta\P$ is \emph{additive} if $\Delta\pi =0$ over all $F$.  Note that
$\Delta\pi$ is affine linear over $F$, and so the face $F$ is additive if and
only if $\Delta\pi(x, y) = 0$ for all $(x, y) \in \verts(F)$. 
It is clear that any subface $E$ of an additive face $F$ ($E \subseteq F$, $E
\in \Delta\P$) is still additive. 
Thus the additivity domain of~$\pi$ can be represented by the list of inclusion-maximal additive faces of $\Delta\P$; see
\cite[Lemma~3.12]{igp_survey}. This list is computed by 
  \sage{generate\_maximal\_additive\_faces($\pi$)}, whose algorithm will be explained in \autoref{sec:algo-maximal-additive-face}.

\textbf{For a discontinuous subadditive function $\pi$}, we say that a face $F
\in \Delta\P$ is \emph{additive} if $F$ is contained in a face $F'
\in \Delta\P$ such that $\Delta\pi_{F'}(x,y) =0$ for any $(x,y) \in F$.%
\footnote{Summarizing the detailed additivity and additivity-in-the-limit 
  situation of the function using the notion of additive faces is justified by 
  \autoref{thm:directly_covered} and \autoref{thm:indirectly_covered}.
}  
Since $\Delta\pi$ is affine linear in the relative interiors of each face of $\Delta\P$, the last condition is equivalent to $\Delta\pi_{F'}(x,y) =0$ for any $(x,y) \in \verts(F)$
. Depending on the dimension of $F$, we do the following.
\begin{enumerate}
\item Let $F$ be a two-dimensional face  of $\Delta\P$. If $\Delta\pi_{F}(x,y) =0$ for any $(x,y) \in \verts(F)$, then $F$ is additive. Visually on the 2d-diagram with cones, each vertex of $F$ has a green cone sitting inside $F$.
\item Let $F$ be a one-dimensional face, i.e., an edge of
  $\Delta\P$. Let $(x_1, y_1), (x_2, y_2)$ be its vertices. Besides
  $F$ itself, there are two other faces $F_1, F_2 \in \Delta\P$ that
  contain $F$. If $\Delta\pi_{F'}(x_1,y_1)=\Delta\pi_{F'}(x_2,y_2) =0$ for $F'
  = F$, $F_1$, or $F_2$, then the edge $F$ is additive. 
\item Let $F$ be a zero-dimensional face of $\Delta\P$, $F = \{(x, y)\}$. If there is a face $F' \in \Delta\P$ such that $(x,y) \in F'$ and $\Delta\pi_{F'}(x,y)=0$, then $F$ is additive.  Visually on the 2d-diagram with cones, the vertex $(x,y)$ is green or there is a green cone pointing at $(x,y)$. 
\end{enumerate}

On the diagrams in \autoref{fig:2d_diagrams_continuous_function} (right) and 
\autoref{fig:2d_diagrams_discontinuous_function} (right), the 
additive faces are shaded in green.  
The projections $p_1(F)$, $p_2(F)$, and $p_3(F)$ of a two-dimensional additive face $F$ are shown as gray shadows on the $x$-, $y$- and $(x+y)$-axes of the diagram, respectively.
These projections become important in the computation of covered intervals
(\autoref{sec:connected-covered-components} below). 


The area of the additive faces of $\pi$ is related to the notion of \emph{merit index} defined by Gomory and Johnson in \cite{tspace}. For a minimal valid function $\pi$, the merit index is defined as twice the area of the additivity domain $\{(x,y)\in [0,1]^2 \colon \Delta\pi(x,y)=0\}$ of $\pi$. This index, available as \sage{merit\_index} in our code, was proposed as a quantitative measure of strength of minimal valid functions. 
For example, the merit index of $\pi=\sage{gmic}(f)$ is $2f^2-2f+1$.

\section{Minimality test}
\label{s:minimality}
The Python function \sage{minimality\_test($\pi$, $f$)} implements a fully automatic test whether a
  given function is a minimal valid function, using the information that the
  described 2-dimensional diagrams visualize.  The algorithm is equivalent to
  the one described, in the setting of discontinuous pseudo-periodic
  superadditive functions, in Richard, Li, and Miller \cite[Theorem
  22]{Richard-Li-Miller-2009:Approximate-Liftings}.
  
Let $\pi$ be a piecewise linear $\Z$-periodic function and let $f \in [0,1]$. 
The command \sage{minimality\_test($\pi$, $f$)} 
verifies the characterization of minimal functions by Gomory--Johnson
\cite{infinite} 
that we mentioned in
the introduction, returning
\sage{True} if and only if the conditions (\ref{eq:minimal:01}--\ref{eq:minimal:subadd})
are satisfied. 

Note that because the given function $\pi$ is piecewise linear and
$\Z$-periodic, the conditions \eqref{eq:minimal:nonneg} and
\eqref{eq:minimal:symm} only need to be checked on the breakpoints (including
the limits) of $\pi$ in $[0,1]$, namely, for $x \in \{x_0, x_1,\dots, x_n,
x_0^+, x_1^+,\dots, x_n^+, x_0^-, x_1^-,\dots, x_n^-\}$. In regards to the
subadditivity condition \eqref{eq:minimal:subadd}, it suffices to check
$\Delta\pi(x,y)\geq 0$ at the vertices $(x,y)$ of the 2-dimensional polyhedral
complex $\Delta\P$ of $\pi$, including the limit values $\Delta\pi_F(x,y)$
when $\pi$ is a discontinuous function. The sign of $\Delta\pi$ is indicated
by colors on the diagram \sage{plot\_2d\_diagram\_with\_cones($\pi$)}: if the
diagram does not contain anything in red, then $\pi$ satisfies the
subadditivity condition~\eqref{eq:minimal:subadd}.  One can further restrict
to the upper-left triangular part of $\Delta\P$ where $x \leq y$, since
$\Delta\pi(x,y)=\Delta\pi(y,x)$ for any $x,y \in \R$. 

It is clear that a minimal valid function $\pi$ satisfies $\pi(f)=1$. If the value of the optional argument \sage{f} is not provided, then \sage{minimality\_test($\pi$)} uses $f =$ \sage{find\_f($\pi$)}, which returns the first breakpoint $x_i \in [0,1]$ of $\pi$ such that $\pi(x_i)=1$. Note that if a minimal valid function $\pi$ has distinct breakpoints $x_i, x_j \in [0,1]$ such that $\pi(x_i)=\pi(x_j)=1$, then by the symmetry condition \eqref{eq:minimal:symm}, there exists $b \in (0,1)$ such that $\pi(b)=0$.  In this case, the following lemma implies that $\pi$ is actually a $\frac1q\Z$-periodic function. Therefore, one can set $f$ to the first $x_i$ with $\pi(x_i)=1$ without loss of generality. 

\begin{lemma}
Let $\pi$ be a piecewise linear $\Z$-periodic function that is minimal valid. If $\pi(b) = 0$ for some $b \not\in \Z$, then $b\in \Q$ and $\pi$ is $\frac1q\Z$-periodic, where $q$ denotes the positive denominator of $b$ written as an irreducible fraction.
\end{lemma}
\begin{proof}
  Since the function $\pi$ is non-negative and subadditive, one can show by induction that $\pi(nb)=0$ for any integer $n$. Suppose that $b \not\in \Q$.
The minimal valid function $\pi$ is $0$ on the set $b\Z$ and is $1$ on the set $f-b\Z$, where both sets are dense in $\R/\Z$. This contradicts the piecewise linearity of $\pi$. Therefore, $b\in\Q$.
Let $p \in \Z$ and $q \in \Z_+$ such that $\gcd(p,q)=1$ and $b =\frac{p}{q}$. 
There exists $r \in \Z_+$ such that $rp \equiv 1 \pmod q$. We have that $0 = r\pi(b) \geq \pi(rb) = \pi(\frac1q)$, and hence $\pi(\frac1q)=0$. Similarly, we have that $\pi(-\frac1q)=0$.
Let $x \in \R$. By subadditivity, $\pi(x) + \pi(\frac1q) \geq \pi(x+\frac1q)$ and $\pi(x+\frac1q) + \pi(-\frac1q) \geq \pi(x)$. Therefore, $\pi(x) = \pi(x+\frac1q)$. We conclude that the function $\pi$ is $\frac1q\Z$-periodic.
\end{proof}

If the \sage{minimality\_test} is called with the optional argument \sage{show\_plots=True} (default: \sage{False}), then it shows the diagram of $\Delta\P$ illustrating the sign of $\Delta\pi(x,y)$ by colors.

\section{Connected components of covered intervals}
\label{sec:connected-covered-components}
An additive face implies, among other things, the important covering (affine
imposing in the terminology of \cite{basu-hildebrand-koeppe:equivariant}) property that we outline in this section. The actual use of such covering property in our code enables a grid-free variant of the algorithmic results in \cite[Theorem 1.5]{basu-hildebrand-koeppe:equivariant}.

Recall that a minimal valid function $\pi$ is said to be \emph{extreme} if it cannot be written as a
convex combination of two other minimal valid functions.
We say that a function~$\tilde\pi$ is an \emph{effective perturbation function} for
the minimal valid function~$\pi$, denoted $\tilde\pi \in \tilde\Pi^{\pi}(\R,\Z)$, 
if there exists $\epsilon>0$ such that $\pi\pm\epsilon\bar\pi$ are minimal
valid functions.\footnote{The space $\tilde\Pi^{\pi}(\R,\Z)$ of effective
  perturbation functions, should not be confounded with $\bar{\Pi}^E(\R,\Z)$,
  the space of perturbation functions with prescribed additivities $E$, 
  defined in \cite{igp_survey}
  as
\[
\bar{\Pi}^E(\R,\Z) = \left\{\bar{\pi} \colon \R \to \R \, \Bigg| \,
\begin{array}{r@{\;}c@{\;}ll}
\bar{\pi}(0) &=& 0 \\
\bar{\pi}(f) &=& 0 \\
\bar{\pi}(x) + \bar{\pi}(y) &=& \bar{\pi}(x+y) & \text{ for all } (x,y) \in E\\
\bar{\pi}(x) &=& \bar{\pi}(x+t) & \text{ for all } x \in \R,\, t \in \Z
\end{array} \right\}.
\]
Let $ E = \{\,(x,y) \mid \Delta\pi(x,y) = 0\,\}$.
Then, by \cite[Theorem 3.13]{igp_survey}, if $\pi$ and $\bar\pi$ are continuous piecewise linear
and $\bar \pi \in \bar{\Pi}^E(\R,\Z)$, then $\bar\pi \in \tilde\Pi^{\pi}(\R,\Z) $.
In general, $\tilde\Pi^{\pi}(\R,\Z)$ is a subspace of $\bar{\Pi}^E(\R,\Z)$.
}
Thus a minimal valid function $\pi$ is extreme if and only if no non-zero effective
perturbation $\tilde{\pi} \in \tilde{\Pi}^{\pi}(\R,\Z)$ exists. 

The key technique for studying the space of effective perturbations is to
analyze the additivity relations.  The foundation of the technique is the
following lemma, which shows that all subadditivity conditions that are tight
(satisfied with equality) for $\pi$ are also tight for an effective
perturbation $\tilde{\pi}$.  This includes additivity in the limit.
\begin{lemma}[{\cite[Lemma 2.7]{basu-hildebrand-koeppe:equivariant}}]
\label{lemma:tight-implies-tight}
Let $\pi$ be a minimal valid function that is piecewise linear over $\P$. Let $F$ be a face of $\Delta\P$ and let $(u, v) \in F$. If $\Delta\pi_F(u,v)=0$, then $\Delta\tilde{\pi}_F(u,v)=0$ for any effective perturbation function $\tilde\pi \in \tilde\Pi^{\pi}(\R,\Z)$.
\end{lemma}
\begin{proof}
Let $\epsilon>0$ such that $\pi^+ = \pi+\epsilon\tilde{\pi}$ and $\pi^-=\pi-\epsilon\tilde{\pi}$ are minimal valid functions.
Since $\pi^+$ and $\pi^-$ are subadditive, we have $\Delta\pi^\pm \geq 0$ and
so $\epsilon \left| \Delta\tilde\pi_F(u, v)\right| \leq \Delta\pi_F(u,v) =
0$.  Thus we have $\Delta\tilde\pi_F(u,v)=0$. 
\end{proof}


We first make use of the additivity relations that are captured by the
two-dimensional additive faces $F$ of $\Delta\P$.
Note that minimal valid functions~$\pi$ and effective perturbation functions
$\tilde{\pi} \in \tilde{\Pi}^{\pi}(\R,\Z)$ are bounded functions.  As such,
they satisfy the regularity assumptions that rule out pathological
solutions to Cauchy's functional equation; see \cite[section 4.1]{igp_survey}.  Thus the following is an immediate
corollary of the convex additivity domain lemma \cite[Theorem
4.3]{igp_survey}, a variant of the celebrated Gomory--Johnson interval lemma.
\begin{theorem}
\label{thm:directly_covered}
Let $\pi$ be a minimal valid function that is piecewise linear over $\P$. Let $F$ be a two-dimensional additive face of $\Delta\P$.  Let $\theta=\pi$ or $\theta =\tilde{\pi} \in \tilde{\Pi}^{\pi}(\R,\Z)$.
Then $\theta$ is affine with the same slope over $\intr(p_1(F))$, $\intr(p_2(F))$, and $\intr(p_3(F))$.\footnote{If the function $\pi$ is continuous, then 
$\theta$ is affine with the same slope over the closed intervals $p_1(F)$, $p_2(F)$, and $p_3(F)$, 
by \cite[Corollary 4.9]{igp_survey}.}
\end{theorem}
In the situation of this result, we say that the intervals $\intr(p_1(F))$,
$\intr(p_2(F))$, and $\intr(p_3(F))$ are \emph{(directly) covered}\footnote{In
  the terminology of \cite{basu-hildebrand-koeppe:equivariant}, these
  intervals are said to be \emph{affine imposing}.} 
and are in the same \emph{connected covered component}\footnote{Connected covered components, extending
  the terminology of \cite{basu-hildebrand-koeppe:equivariant}, are simply
  collections of intervals on which an effective perturbation function is affine with the same
  slope.  This notion of connectivity is unrelated to that in the topology of
  the real line, but is understood in a graph-theoretic sense. 
  In the grid-based algorithm in~\cite{basu-hildebrand-koeppe:equivariant}, 
  notation for an explicit graph whose nodes are the breakpoint intervals is
  introduced.
  In the present paper, we do not introduce such a graph explicitly.}.
If an interval is contained in a connected covered component, then all of its sub-intervals are also contained in this connected covered component. In particular, any sub-interval of a covered interval is also covered.

\smallskip

Now let $F$ be a one-dimensional additive face (edge) of
$\Delta\P$. The edge $F$ can be vertical, horizontal, or
diagonal. Then two of the projections $\intr(p_1(F)), \intr(p_2(F))$ and
$\intr(p_3(F))$\footnote{The closed intervals $p_1(F), p_2(F)$ and $p_3(F)$
  are considered when $\pi$ is a continuous function.} are
one-dimensional (proper intervals), whereas the third projection is a
singleton. 
The following theorem holds, 
which is akin to 
\cite[Lemma 4.5]{basu-hildebrand-koeppe:equivariant}.

\begin{theorem}
\label{thm:indirectly_covered}
Let $\pi$ be a minimal valid function that is piecewise linear over $\P$. 
Assume that $\pi$ is at least one-sided continuous at the
origin.\footnote{Note that \autoref{thm:indirectly_covered} holds as well when
  the function $\pi$ is two-sided discontinuous at the origin. If all the
  breakpoints of $\pi$ are rational numbers, it is justified by \cite[Lemma
  4.5]{basu-hildebrand-koeppe:equivariant}. In the general case, the proof of
  the theorem needs a stronger regularity lemma than
  \autoref{corollay:perturbation-lim-exist}, which we defer to the forthcoming
  paper \cite{koeppe-zhou:crazy-perturbation}.}
Let $F$ be a one-dimensional additive face (edge) of $\Delta\P$.
Let $\{i,j\} \subset \{1,2,3\}$ such that $p_i(F)$ and $p_j(F)$ are proper intervals. 
Let $E \subseteq F$ be a sub-interval. For $\theta=\pi$ or $\theta =\tilde{\pi} \in \tilde{\Pi}^{\pi}(\R,\Z)$, if $\theta$ is affine in $I = \intr(p_i(E))$, then $\theta$ is affine in $I' = \intr(p_j(E))$ as well with the same slope.
\end{theorem}
In the situation of the theorem, the two proper intervals $p_i(F)$ and
$p_j(F)$ are said to be \emph{connected} through a translation 
(when $F$ is a vertical or horizontal edge) or through a reflection (when $F$
is a diagonal edge). An interval $I'$ that is connected to a covered interval
$I$ said to be \textit{(indirectly) covered} and in the same connected
component as $I$.

Before proving \autoref{thm:indirectly_covered}, we first discuss some important regularity results.
Assume that $\pi$ is at least one-sided continuous at the origin, from the left or from the right. 
Dey--Richard--Li--Miller \cite[Theorem 2]{dey1} (see also \cite[Lemma 2.11 (v)]{igp_survey}) showed that 
any effective perturbation function $\tilde\pi \in \tilde\Pi^{\pi}(\R,\Z)$ is continuous at all points at which $\pi$ is continuous. In fact, the result can be strengthened to Lipschitz continuous, as follows.
\begin{lemma}
\label{lemma:perturbation-lipschitz-continuous}
Let $\pi$ be a piecewise linear minimal valid function that is continuous from the right at $0$ or continuous from the left at $1$. If $\pi$ is continuous on a proper interval $I \subseteq [0,1]$, then for any $\tilde{\pi} \in \tilde{\Pi}^{\pi}(\R,\Z)$ we have that $\tilde{\pi}$ is Lipschitz continuous on the interval $I$.
\end{lemma} 
\begin{proof}
Without loss of generality, we assume that $\pi$ is continuous from the right at $0$. 
Then,  since $\pi$ is also piecewise linear and $\pi(0)=0$, 
there exist positive $s, b\in \R$ such that $\pi(x) = s x$ for $x \in [0, 2b]$. 
Let $\tilde{\pi} \in \tilde{\Pi}^\pi(\R,\Z)$  be an effective perturbation.
By definition, there exists $\epsilon>0$ such that $\pi^+ = \pi+\epsilon\tilde{\pi}$ and $\pi^-=\pi-\epsilon\tilde{\pi}$ are minimal valid functions.
For all $x,y \in [0,b]$,  we have that $\pi(x)+\pi(y)=\pi(x+y)$. 
Since the functions $\pi^{+}$ and $\pi^{-}$ are subadditive,
$\tilde\pi(x)+\tilde\pi(y)=\tilde\pi(x+y)$ for all $x,y \in [0,b]$. 
Note that $\tilde\pi(0)=0$. By the Gomory--Johnson Interval Lemma
\cite[Lemma~4.1]{igp_survey}, there exists $\tilde{s} \in \R$ such that
$\tilde{\pi}(x) = \tilde{s}x$ for $x \in [0,b]$.  Then $\pi^+$ and $\pi^-$ have slopes $s^+ := s+\epsilon\tilde s$ and $s^- := s-\epsilon\tilde s$ on $[0, b]$, respectively. 

Let $I \subseteq [0,1]$ be an interval where $\pi$ is continuous.  Let $x, y \in I$ such that $x>y$. 
Then there exists $s_I \in \R$ such that $\pi(x)-\pi(y) \geq s_I(x-y)$, since
$\pi$ is piecewise linear and continuous on~$I$. 
By subadditivity, we have $\pi^+(x)-\pi^+(y) \leq s^+(x-y)$ and  $\pi^-(x)-\pi^-(y) \leq s^-(x-y)$.
It follows from $\epsilon\tilde{\pi} = \pi^+ - \pi = \pi - \pi^-$
that $(s_I-s^-)(x-y) \leq \epsilon(\tilde{\pi}(x)-\tilde{\pi}(y))\leq (s^+ - s_I)(x-y)$.
Therefore, $\left| \tilde\pi(x)-\tilde\pi(y)\right| \leq C \left| x-y\right|$, where $C = \frac{1}{\epsilon}\max(\left| s^+ - s_I\right|, \left| s^- - s_I \right|)$ is a constant independent of $x, y \in I$.

We conclude that $\tilde{\pi}$  is Lipschitz continuous on the interval~$I$.
\end{proof}
Since the function $\pi$ is continuous on each interval $(x_i, x_{i+1})$ between two consecutive breakpoints, \autoref{lemma:perturbation-lipschitz-continuous} implies the important corollary below.
\begin{corollary}
\label{corollay:perturbation-lim-exist}
Let $\pi$ be a piecewise linear minimal valid function that is continuous from the right at $0$ or continuous from the left at $1$. Let $\tilde{\pi} \in \tilde{\Pi}^{\pi}(\R,\Z)$. Then the limit values $\tilde\pi(x^-)$ and $\tilde\pi(x^+)$  are well defined for any $x \in [0,1]$.
\end{corollary}

We now provide a proof to \autoref{thm:indirectly_covered}.
 
\begin{proof}[Proof of \autoref{thm:indirectly_covered}]
Assume that $F$ is an horizontal additive edge of $\Delta\P$, with
$p_2(F)=\{t\}$ and $p_1(F), p_3(F)$ being proper intervals. There exists a
face $F' \in \Delta\P$ containing $F$, such that $\Delta\pi_{F'}(x, t)=0$ for
any $x \in p_1(F)$. Consider $\theta=\pi$ or $\theta =\tilde{\pi}$. By
\autoref{lemma:tight-implies-tight}, $\Delta\theta_{F'}(x,t)=0$ for $x \in
p_1(F)$. By \autoref{corollay:perturbation-lim-exist}, the values $\theta(t),
\theta(t^+), \theta(t^-)$ are well defined, and hence the limit value $\ell :=\lim_{\substack{y \to t \\ y \in \relint(p_2(F'))}} \theta(y)$ exists.
 Assume that the function $\theta$ is affine in $I = \intr(p_1(E))$ for some interval $E \subseteq F$ with slope $c$, i.e., $\theta(x) = cx+b$ for $x \in I$, where $b, c \in \R$.  Then for any $x \in I \subseteq \intr(p_1(F))$, we have $0 = \Delta\theta_{F'}(x,t) = \theta(x)+\ell - \theta(x+t)$. Denote $z=x+t$ and $I' = \intr(p_3(E))$, then $x \in I$ for $z \in I'$. Therefore, $\theta(z) =  \theta(z-t)+\ell=c(z-t)+b+\ell$ for $z \in I'$. We see that $\theta(z)$ is affine with the same slope $c$ on the interval $I'$ as well.  Now assume that $\theta(z) = c'z+b'$ for $z \in \intr(p_3(E'))$, where $E' \subseteq F$ and $b', c' \in \R$.  Then by the same argument, $\theta(x) = c'(x+t)+b'-\ell$ for $x \in \intr(p_1(E'))$, showing that $\theta$ is affine in $\intr(p_1(E'))$ as well with the same slope. 

If $F$ is an vertical additive edge of $\Delta\P$, the result follows from swapping $x$ and $y$. 

Assume that $F$ is a diagonal additive edge of $\Delta\P$, with $p_3(F)=\{r\}$ and $p_i(F), p_j(F)$ ($\{i,j\}=\{1,2\}$) being proper intervals.  There exists a face $F' \in \Delta\P$ containing $F$, such that $\Delta\pi_{F'}(x, y)=0$ for any $(x,y)\in F$. Consider $\theta=\pi$ or $\theta =\tilde{\pi}$. By \autoref{lemma:tight-implies-tight}, $\Delta\theta_{F'}(x,y)=0$ for $(x, y) \in F$. \autoref{corollay:perturbation-lim-exist} implies that the values $\theta(r), \theta(r^+), \theta(r^-)$ are well defined, and thus, $\ell :=\lim_{\substack{z \to r \\ z \in \relint(p_3(F'))}} \theta(z) $ exists. We have that $\Delta\theta_{F'}(x,y)= \theta(x)+\theta(y) -\ell = 0$, hence $\theta(x)+\theta(y) =\ell$ for $(x, y)\in \relint(F)$. Therefore, if $\theta$ is affine in $I = \intr(p_i(E))$ for some interval $E \subseteq F$, then $\theta$ is affine also in $I' = \intr(p_j(E))$ with the same slope.
\end{proof}

The connected components of covered intervals of $\pi$ are computed in two phases, by calling \sage{generate\_covered\_components($\pi$)}.
Recall that within a connected covered component, the function~$\pi$, and any perturbation $\tilde\pi$, is affine linear with the same slope. 

In phase one, the program computes the connected components of directly covered intervals of $\pi$ according to \autoref{thm:directly_covered}, using the two-dimensional additive faces of $\Delta\P$. 
Let $F \in \Delta\P$ be a two-dimensional additive face. Then the intervals  $\intr(p_1(F)), \intr(p_2(F))$ and $\intr(p_3(F)) \bmod 1$ are in the same connected covered component
$\mathcal{C} = \intr(p_1(F)) \cup \intr(p_2(F)) \cup \bigl( \intr(p_3(F)) \bmod 1 \bigl)$.
(Consider $\mathcal{C} = p_1(F) \cup p_2(F) \cup \bigl(p_3(F) \bmod 1 \bigr)$ instead, when $\pi$ is a continuous function.)
If another connected covered component $\mathcal{C}'$ satisfies that $\intr(\mathcal{C}) \cap \intr(\mathcal{C}') \neq \emptyset$, then the program merges $\mathcal{C}'$ into $\mathcal{C}$ to form a large connected covered component  $\mathcal{C}_{\mathrm{new}} \leftarrow \mathcal{C} \cup \mathcal{C}'$, and sets $\mathcal{C}' \leftarrow \emptyset$.

\begin{figure}[h]
\centering
\includegraphics[width=.32\linewidth]{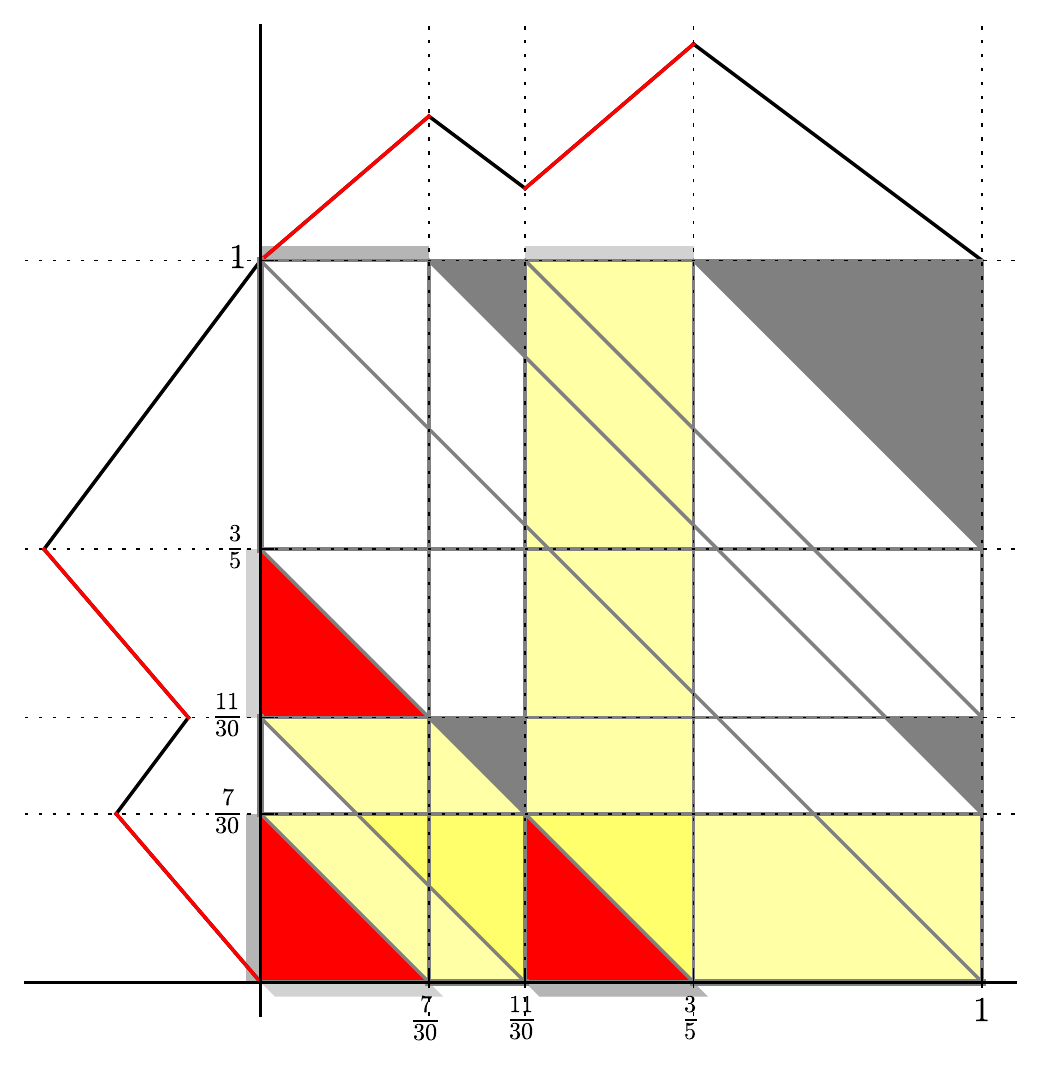}
\includegraphics[width=.32\linewidth]{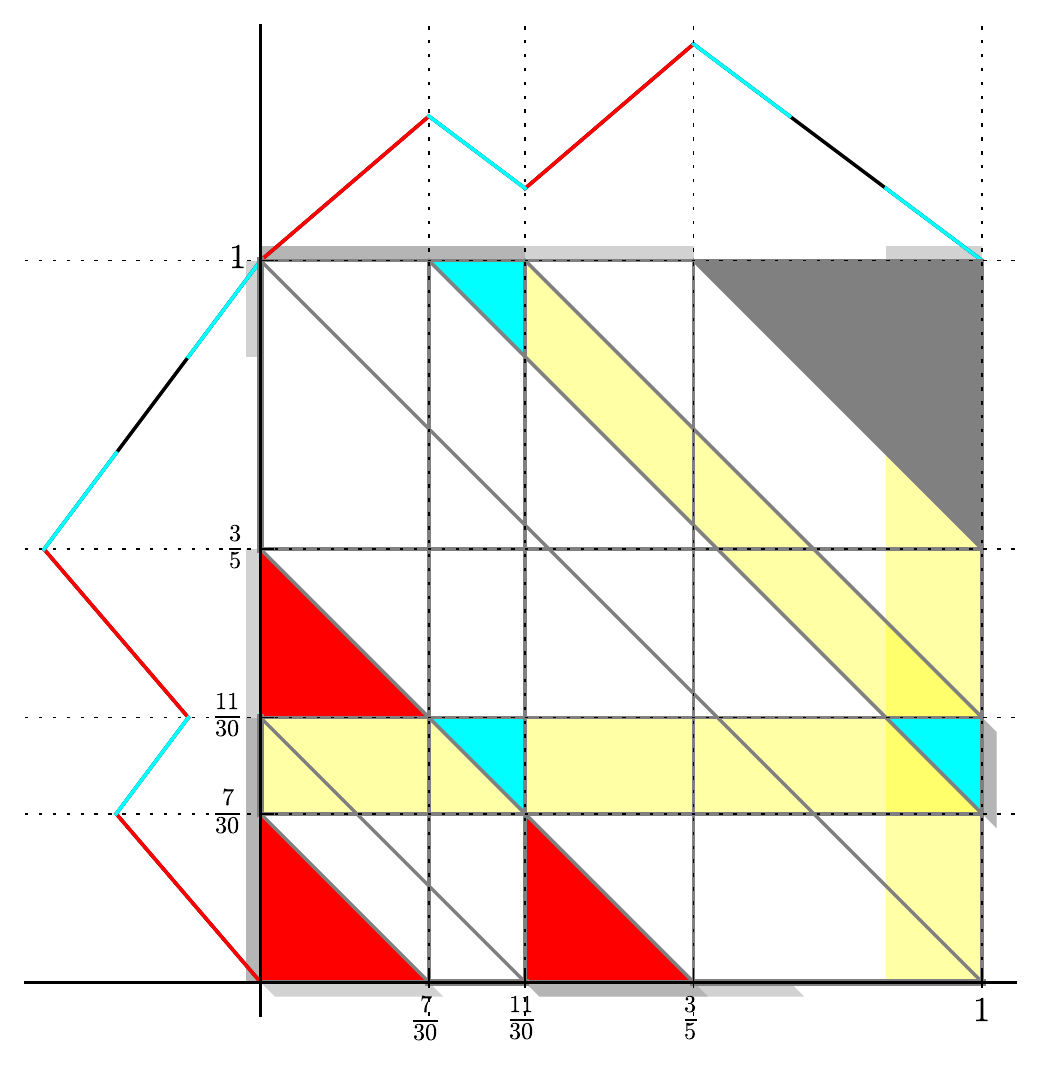}
\includegraphics[width=.32\linewidth]{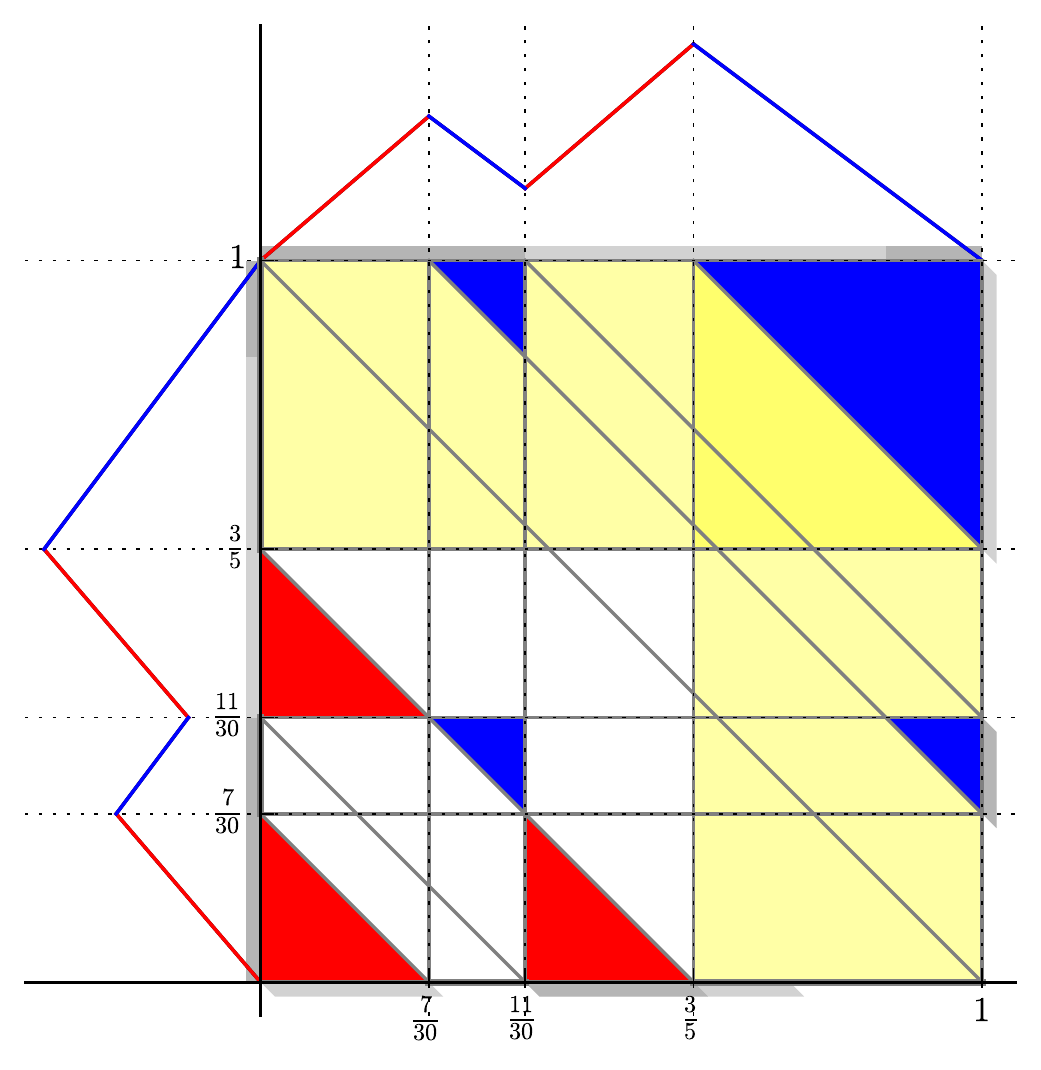}
\caption{Compute the (directly) covered intervals for \sage{$\pi =$ gj\_2\_slope(3/5,1/3)}.}
\label{fig:compute_covered_intervals_cont}
\end{figure}

\begin{example}
\autoref{fig:compute_covered_intervals_cont} shows the computation for \sage{$\pi=$ gj\_2\_slope(3/5,1/3)} step by step, where colors indicate membership to connected covered components. 
In \autoref{fig:compute_covered_intervals_cont}--(1), the additive face $F_1$ with $p_1(F_1)=[\frac{11}{30}, \frac35]$, $p_2(F_1) = [0, \frac{7}{30}]$ and $p_3(F_1) =[\frac{11}{30}, \frac35]$ is considered. Its three projections (indicated by yellow strips on $\Delta\P$) are directly covered (indicated by red color on the function graphs), and they form the connected covered component $\mathcal{C}_1 = [0, \frac{7}{30}] \cup  [\frac{11}{30}, \frac35]$. In  \autoref{fig:compute_covered_intervals_cont}--(2), the connected covered component $\mathcal{C}_2 = [\frac{7}{30}, \frac{11}{30}] \cup  [\frac{13}{15}, 1]$ (with light blue color) is formed by considering the additive face $F_2 = F([\frac{13}{15}, 1], [\frac{7}{30}, \frac{11}{30}] , [\frac{13}{15}, 1])$.  In  \autoref{fig:compute_covered_intervals_cont}--(3), the additive face $F_3$ with $p_1(F_3)=p_2(F_3)=p_3(F_3) \bmod 1=[\frac35, 1]$ is considered, and the connected covered component $\mathcal{C}_3 = [\frac35, 1]$ is formed. Since $\mathcal{C}_3$ and $\mathcal{C}_2$ overlap, they are merged into one large connected component $\mathcal{C}_{\mathrm{blue}}=\mathcal{C}_2 \cup \mathcal{C}_3$. Therefore, all intervals are covered, and the connected components are $\mathcal{C}_\mathrm{red} = [0, \frac{7}{30}] \cup  [\frac{11}{30}, \frac35]$ and  $\mathcal{C}_{\mathrm{blue}}=[\frac{7}{30}, \frac{11}{30}] \cup [\frac35, 1]$.
\end{example}

In phase two, the program computes indirectly covered intervals and merges connected covered components according to \autoref{thm:indirectly_covered}, using one-dimensional additive faces of $\Delta\P$, until no further change is possible.  We explain this process by the example below.

\begin{figure}[h]
\centering
\includegraphics[width=.32\linewidth]{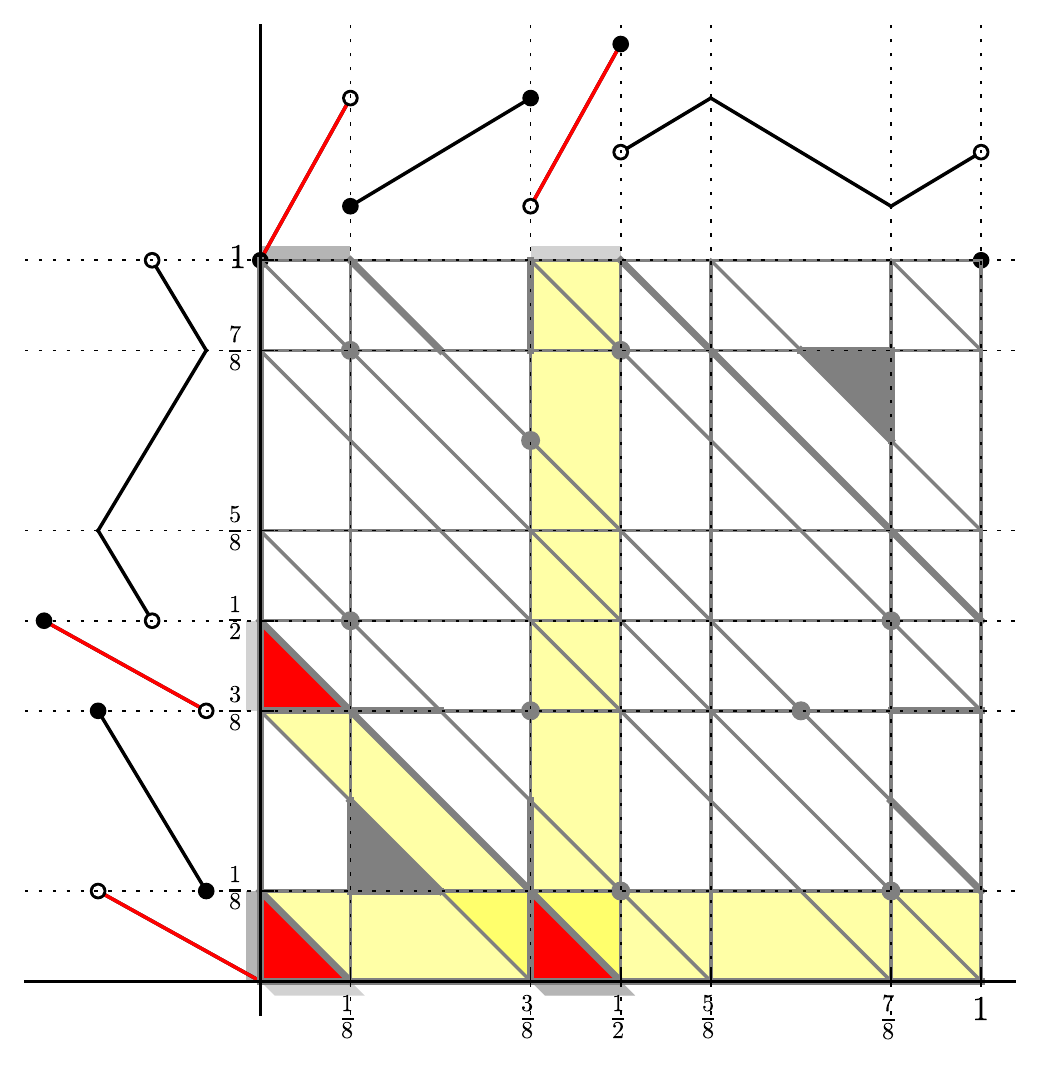}
\includegraphics[width=.32\linewidth]{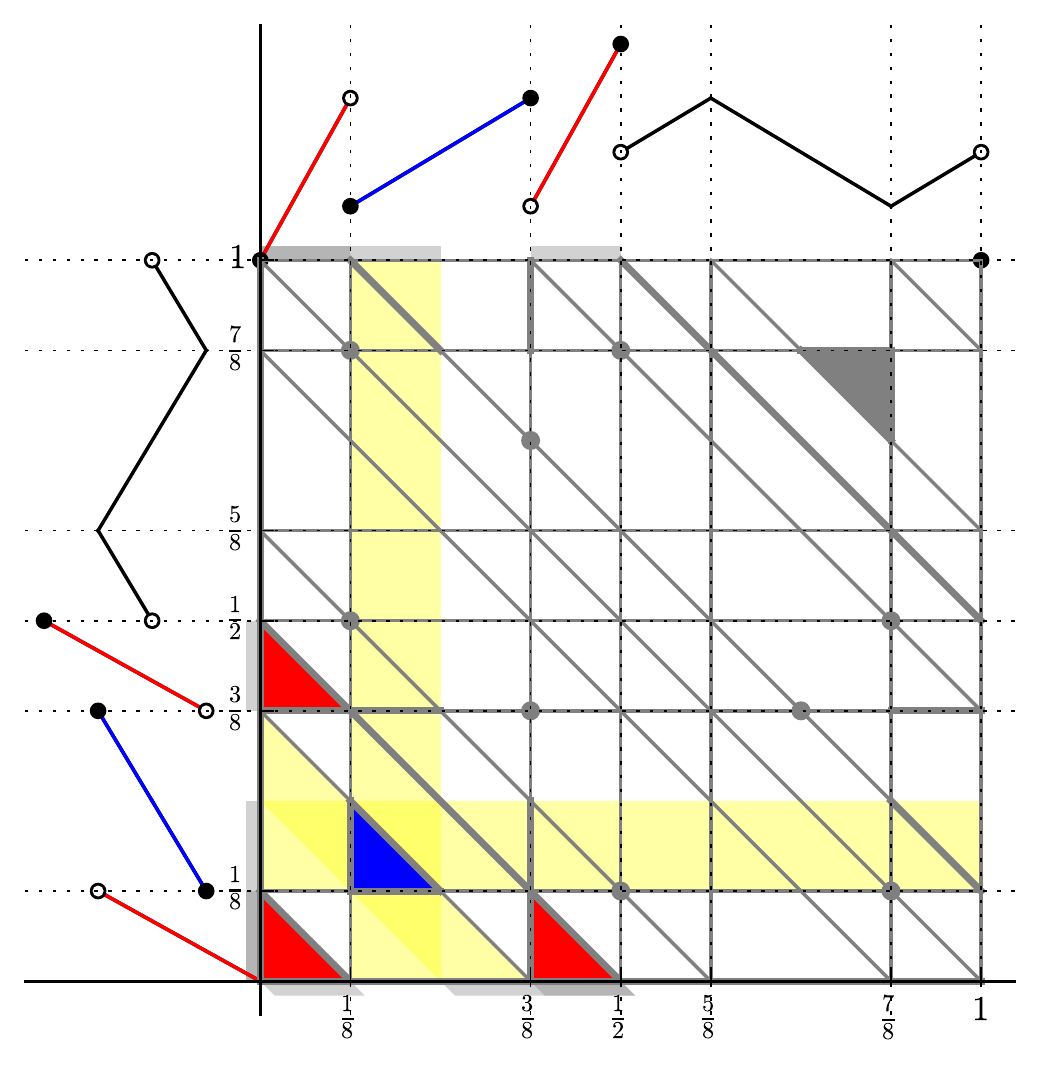}
\includegraphics[width=.32\linewidth]{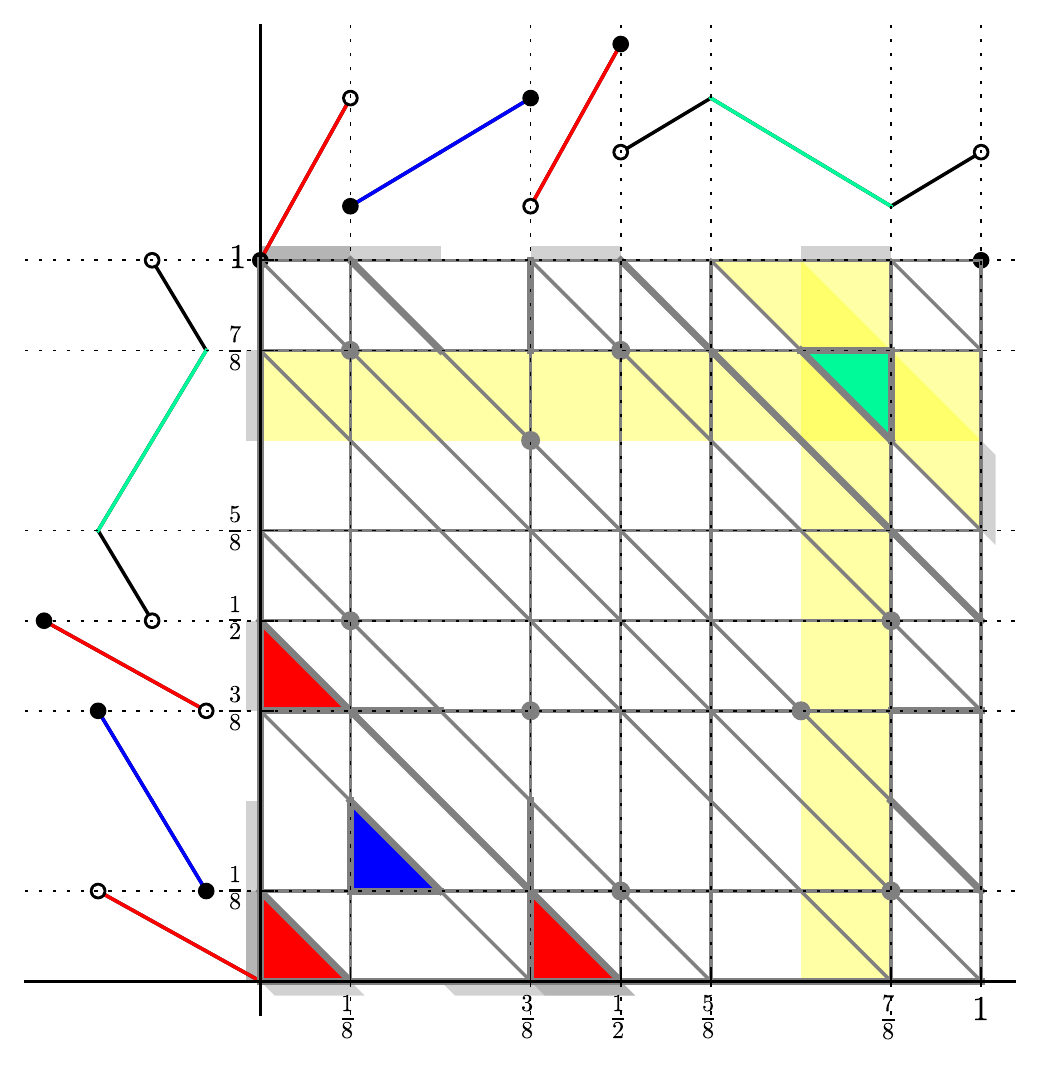}
\\
\includegraphics[width=.32\linewidth]{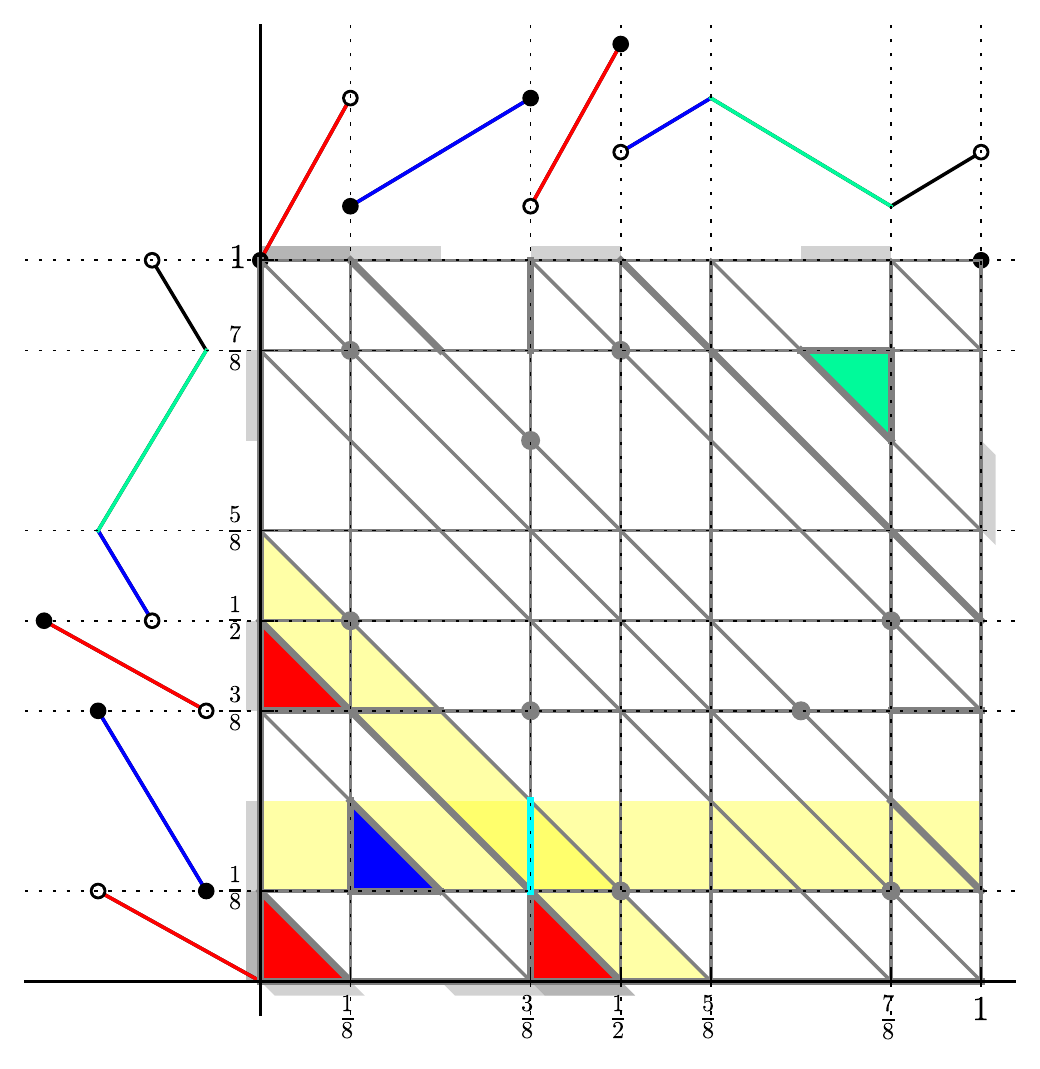}
\includegraphics[width=.32\linewidth]{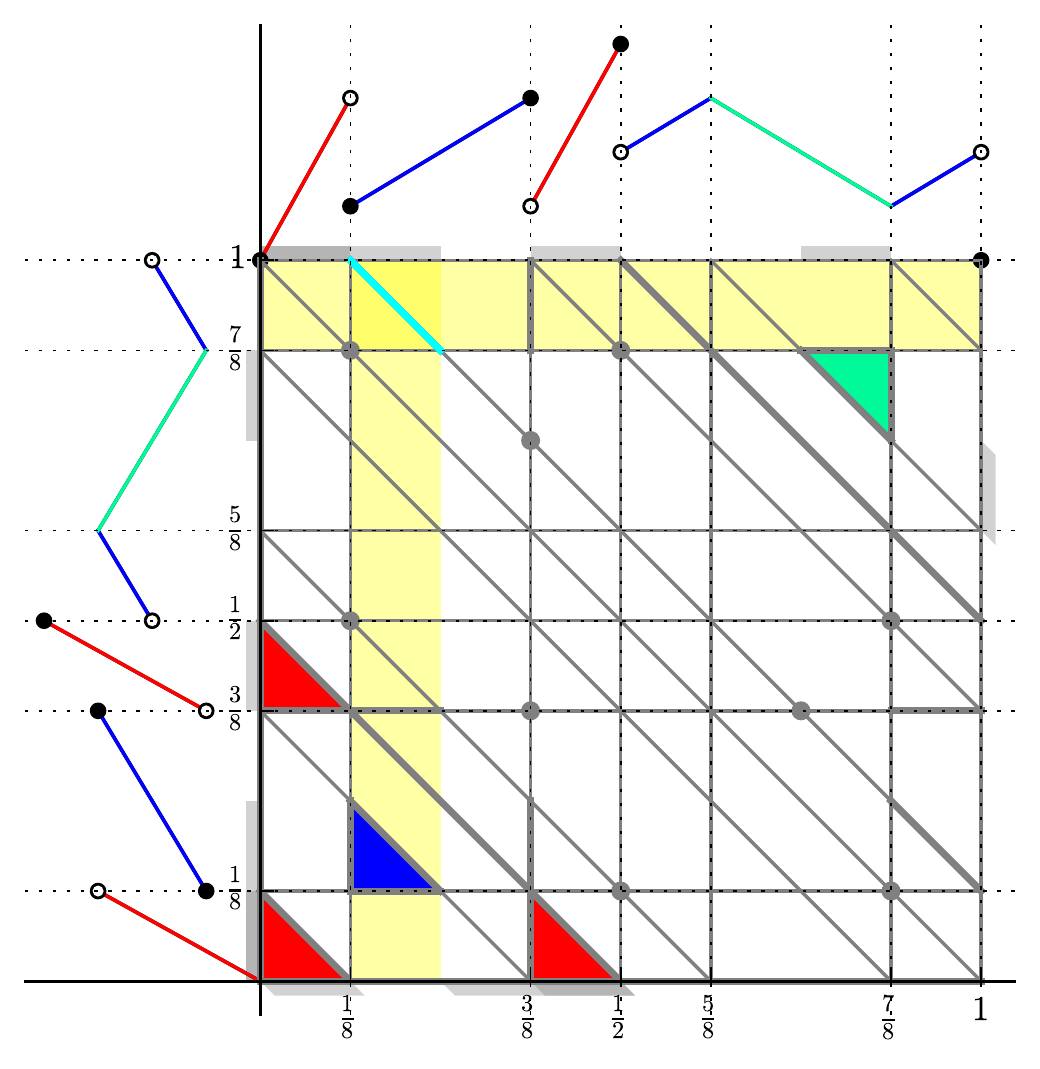}
\includegraphics[width=.32\linewidth]{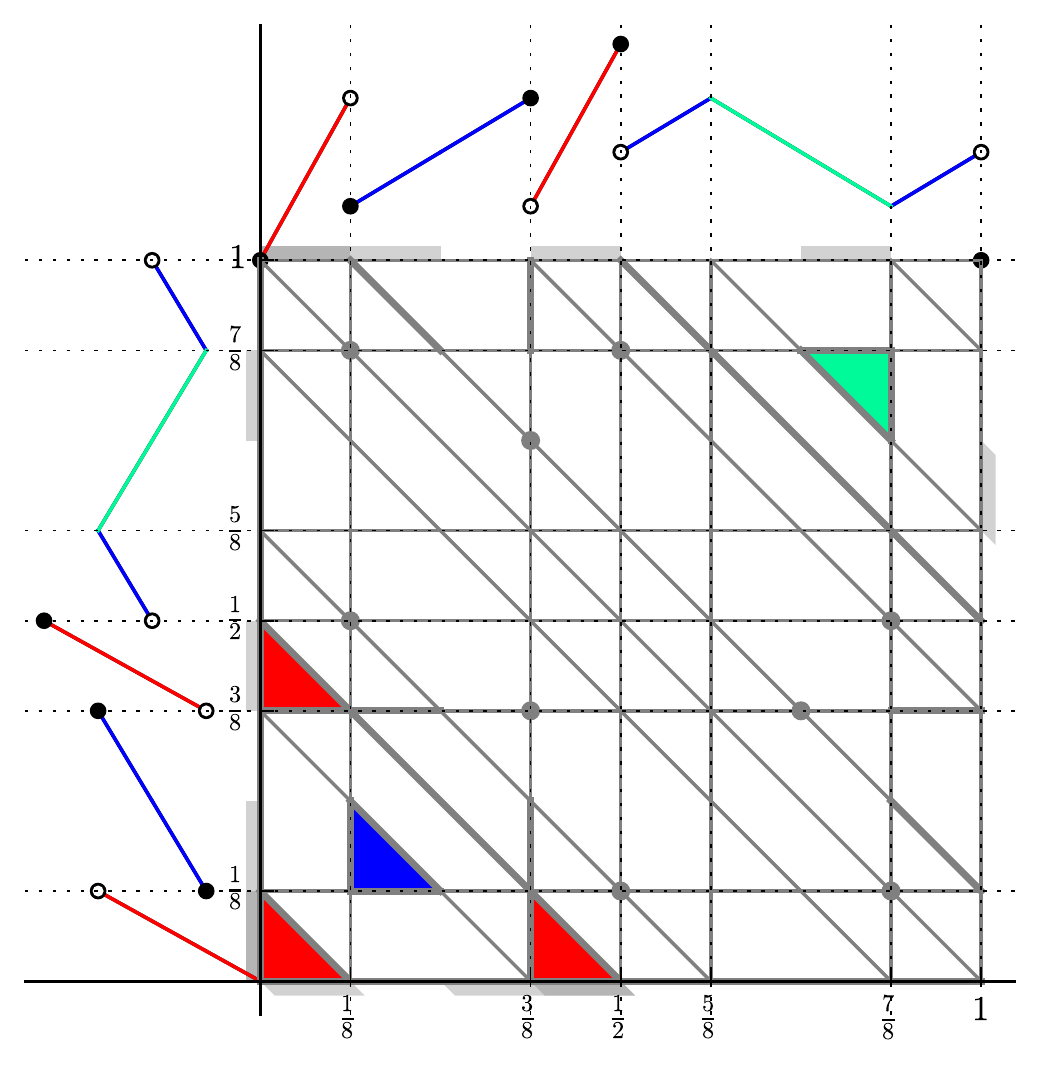}
\caption{Compute the (directly and indirectly) covered intervals for \sage{$\pi=$  hildebrand\_discont\_3\_slope\_1()}}
\label{fig:compute_covered_intervals_disc}
\end{figure} 

\begin{example}
\label{ex:connected_components_disc}
\autoref{fig:compute_covered_intervals_disc} illustrates the computation of connected covered components for the discontinuous function \sage{$\pi$ = hildebrand\_discont\_3\_slope\_1()}. 
In \autoref{fig:compute_covered_intervals_disc}--(1) to (3), the directly covered intervals are computed. They form the connected components $\mathcal{C}_{\mathrm{red}} = (0, \tfrac18) \cup (\tfrac38, \tfrac12), \mathcal{C}_{\mathrm{blue}} = (\tfrac18, \tfrac14) \cup (\tfrac14, \tfrac38)$ and $\mathcal{C}_{\mathrm{green}} = (\tfrac58, \tfrac34) \cup (\tfrac34, \tfrac78)$.
In \autoref{fig:compute_covered_intervals_disc}--(4), the (light blue) vertical additive edge $F(\{\frac38\}, [\frac18, \frac14], [\frac12, \frac58])$ is considered. Its projections $\intr(p_2(F))=(\frac18, \frac14) $ and $\intr(p_3(F))=(\frac12, \frac58)$ are connected through a translation. Since the interval $(\frac18, \frac14)$ was covered in the second sub-figure (as indicated by blue color), the other interval $(\frac12, \frac58)$ which was previously uncovered (as indicated by black color in \autoref{fig:compute_covered_intervals_disc}--(3)), becomes (indirectly) covered (as indicated by blue color in \autoref{fig:compute_covered_intervals_disc}--(4)).
Therefore, the connected covered component $\mathcal{C}_{\mathrm{blue}}$ grows to $(\tfrac18, \tfrac14) \cup (\tfrac14, \tfrac38) \cup (\tfrac12, \tfrac58)$. In \autoref{fig:compute_covered_intervals_disc}--(5), the (light blue) diagonal additive edge $F([\frac18, \frac14], [\frac78, 1], \{\frac98\})$ is considered, and the interval $(\tfrac78, 1)$ is indirectly covered (as indicated by blue color). This concludes the computation of the connected covered components of $\pi$. They are $\mathcal{C}_{\mathrm{red}} =(0, \tfrac18) \cup (\tfrac38, \tfrac12),  \mathcal{C}_{\mathrm{blue}} = (\tfrac18, \tfrac14) \cup (\tfrac14, \tfrac38) \cup (\tfrac12, \tfrac58)  \cup (\tfrac78, 1)$ and $\mathcal{C}_{\mathrm{green}} = (\tfrac58, \tfrac34)\cup (\tfrac34, \tfrac78)$.
\end{example}

The computation of connected components of covered intervals terminates in a finite number of steps for a piecewise linear function $\pi$ with rational breakpoints. For functions with irrational breakpoints, the finiteness of the procedure remains an open question.

\section{Extremality test}
\label{s:extremality}

The command \sage{extremality\_test($\pi$)} implements a grid-free generalization
  of the automatic extremality test from
  \cite{basu-hildebrand-koeppe:equivariant}, which is suitable also for
  piecewise linear functions with rational breakpoints that have huge
  denominators.  Its support for functions with algebraic irrational
  breakpoints such as \sagefunc{bhk_irrational}
  \cite[section~5]{basu-hildebrand-koeppe:equivariant}
  will be described in the papers
  \cite{koeppe-zhou:crazy-perturbation,koeppe-zhou:algo-paper}.  
  
Given a $\Z$-periodic piecewise linear function $\pi$ with rational
breakpoints, \sage{extremality\_test($\pi$)} first checks whether the
conditions for minimality are satisfied, by calling
\sage{minimality\_test($\pi$)}. If $\pi$ is not minimal valid, then
\sage{extremality\_test($\pi$)} returns \sage{False}.  
Otherwise it proceeds, using the following theorem.

\begin{theorem}[{rephrased from \cite[Lemma 4.8]{basu-hildebrand-koeppe:equivariant}}] 
\label{thm:extreme-all-covered}
 Let $\pi$ be a piecewise linear minimal valid function with rational breakpoints. If $\pi$ is an extreme function, then the union of closures of the covered intervals of $\pi$ is equal to the whole interval $[0,1]$.
\end{theorem}
Assume that the function $\pi$ 
is minimal valid. The command \sage{extremality\_test($\pi$)} first checks whether the whole
interval $[0,1]$ is covered by the union of closures of the connected covered
components returned by \sage{generate\_covered\_intervals($\pi$)}. If there is
an uncovered interval, then, assuming that $\pi$ be a piecewise linear
function with rational breakpoints,  the extremality test returns \sage{False},
which is justified by \autoref{thm:extreme-all-covered}.  
Our code also constructs an effective perturbation~$\tilde\pi$ in this case.
However, we defer an explanation of the construction of this ``equivariant'' \cite{basu-hildebrand-koeppe:equivariant}
perturbation, as well the discussion
of the case of functions with irrational breakpoints, to the paper~\cite{koeppe-zhou:algo-paper}.

In the following, we assume that there is no uncovered interval. Let $\mathcal{C}_1, \mathcal{C}_2, \dots, \mathcal{C}_k$ be the connected covered components of $\pi$. Let $\tilde\pi \in \tilde\Pi^{\pi}(\R,\Z)$ be an effective perturbation function. Then $\tilde\pi$ is affine linear with the same slope, say $s_i$, on the $i$-th connected covered component $\mathcal{C}_i$ for $i = 1, 2, \dots, k$.
We also assume that $\pi$ is at least one-sided continuous at the origin. Then
by \autoref{lemma:perturbation-lipschitz-continuous}, the perturbation
$\tilde\pi$ can only be discontinuous at the points where $\pi$ is
discontinuous. Let the variables $d_j$ ($j=1, 2, \dots, m)$ denote the changes
of the value of $\tilde\pi$ at the $m$ discontinuity points of $\pi$. 
In other words, the variables $d_j$ denote jumps $\tilde\pi(x) - \tilde\pi(x^-)$ when $\pi$ is discontinuous at $x$
on the left, or $\tilde\pi(x^+)-\tilde\pi(x)$ when $\pi$ is discontinuous at
$x$ on the right, where $\tilde\pi(x^-)$ and $\tilde\pi(x^+)$ are well defined by \autoref{corollay:perturbation-lim-exist}.
Since $\bigcup_{i=1}^k \cl(\mathcal{C}_i) = [0,1]$, we have that $\tilde\pi$
is a piecewise linear function with breakpoints in $\{x_0, x_1, \dots, x_n\}$,
which can be considered as a symbolic function over the slope parameters
$s_i$ ($i=1, 2, \dots, k$) and jump parameters $d_j$ ($j =1,
2, \dots, m$). In fact for any $x\in [0,1]$, the value $\tilde{\pi}(x)$ is
linearly determined by the variables $s_i$ and $d_j$. We can define a
vector-valued piecewise linear function $g \colon [0,1] \to \R^{k+m}$, 
so that $\tilde\pi(x) = g(x)\cdot (s_1, s_2, \dots, s_k, d_1, d_2,\dots,
d_m)$. This function $g$ can be set up using the command
$\sage{generate\_symbolic}$ in our code. 

The perturbation $\tilde\pi$ satisfies that $\tilde\pi(0)=\tilde\pi(f)=\tilde\pi(1)=0$, and 
all the additivity relations (including additivity-in-the-limit) that $\pi$ has, by \autoref{lemma:tight-implies-tight}. Therefore, we have a system of linear equations for the variables $(s_1, s_2, \dots, s_k, d_1, d_2,\dots, d_m)$, which admits a trivial solution $(0,0,\dots, 0)$. By \cite[Theorem 4.11]{basu-hildebrand-koeppe:equivariant}, we know that $\pi$ is an extreme function if and only if $(0,0,\dots, 0)$ is the unique solution to this system. The command \sage{extremality\_test($\pi$)} solves the above finite-dimensional linear system using linear algebra. It returns \sage{True} if the solution space has dimension $0$ or \sage{False} otherwise.

\begin{example} 
Let $\pi=\sage{hildebrand\underscore discont\underscore 3\underscore slope\underscore 1()}$ with $f = \frac12$. As computed in the \autoref{ex:connected_components_disc}, $\pi$ has three connected covered components $\mathcal{C}_{1} =(0, \tfrac18) \cup (\tfrac38, \tfrac12), \mathcal{C}_{2} = (\tfrac18, \tfrac14) \cup (\tfrac14, \tfrac38) \cup (\tfrac12, \tfrac58)  \cup (\tfrac78, 1)$ and $\mathcal{C}_{3} = (\tfrac58, \tfrac34)\cup (\tfrac34, \tfrac78)$.  An effective perturbation $\tilde\pi \in \tilde\Pi^{\pi}(\R,\Z)$ is affine linear on $\mathcal{C}_i$ ($i=1,2,3$). Let $s_i$ denote the slope value of $\tilde\pi$ on $\mathcal{C}_{i}$, for $i=1,2,3$. The function $\pi$ is discontinuous at the points $\frac18^-, \frac38^+, \frac12^+$ and $1^-$. Let $d_1 = \tilde\pi(\frac18)-\tilde\pi(\frac18^-)$, $d_2= \tilde\pi(\frac12^+)-\tilde\pi(\frac12)$, $d_3 = \tilde\pi(\frac38^+)-\tilde\pi(\frac38)$ and $d_4 = \tilde\pi(1)-\tilde\pi(1^-)$. We know that $\tilde\pi(0)=0$, $\tilde\pi(f)=0$ and $\tilde\pi(1)=0$. Define the piecewise linear function $g\colon [0,1]\to \R^7$, such that
\[
 g(x) = 
 \begin{cases} 
   (1,0,0,0,0,0,0)x & \text{if } 0 \leq x < \frac18 \\
   (0,1,0,0,0,0,0)x + (\frac18,-\frac18,0,1,0,0,0) & \text{if } \frac18 \leq x \leq \frac38 \\
   (1,0,0,0,0,0,0)x - (\frac14,-\frac14,0,-1,-1,0,0) & \text{if } \frac38 < x \leq \frac12 \\
   (0,1,0,0,0,0,0)x + (\frac14,-\frac14,0,1,1,1,0) & \text{if } \frac12 < x \leq \frac58 \\
   (0,0,1,0,0,0,0)x + (\frac14,\frac38,-\frac58,1,1,1,0) & \text{if } \frac58 < x \leq \frac78 \\
   (0,1,0,0,0,0,0)x + (\frac14,-\frac12,\frac14,1,1,1,0) & \text{if } \frac78 < x <1 \\
   (\frac14,\frac12,\frac14,1,1,1,1) & \text{if } x = 1.
  \end{cases}
\]
This function $g$ can be obtained by calling

\begin{scriptsize}
  \begin{tabular}{@{}p{\linewidth}@{}}
	\begin{verbatim}
    h = hildebrand_discont_3_slope_1();
    components = generate_covered_intervals(h);
    g = generate_symbolic(h, components, field=QQ).
	\end{verbatim}
  \end{tabular}
\end{scriptsize}
Then $\pi(x) = g(x) \cdot (s_1, s_2, s_3, d_1, d_2, d_3, d_4)$, for $0 \leq x \leq 1$.
It follows from the symmetry condition that $d_1 = d_3$ and $d_2 = d_4$.
By considering the additive limiting cones on the diagram of $\Delta\P$ (see \autoref{fig:2d_diagrams_discontinuous_function}--left) and by \autoref{lemma:tight-implies-tight}, we obtain that $\tilde\pi(\frac18^+)+\tilde\pi(\frac12^+)=\tilde\pi(\frac58^+)$, $\tilde\pi(\frac78^-)+\tilde\pi(\frac78^-)=\tilde\pi(\frac74^-)$ and $\tilde\pi(0^-)+\tilde\pi(\frac18^+)=\tilde\pi(\frac18^-)$. These additivity relations, along with $\tilde\pi(f)=0$ and $\tilde\pi(1)=0$, imply the linear system
\[
\begin{pmatrix}
\tfrac14 & \tfrac14 & 0 & 1 & 1 & 0 & 0\\
\tfrac14 & \tfrac12 & \tfrac14 & 1 & 1 & 1 & 1 \\
0 & 0 & 0 & 1 & 0 & -1 & 0 \\
0 & 0 & 0 & 0 & 1 & 0 & -1 \\
\tfrac18 & -\tfrac18 & 0 & 1 & 0 & 0 & 0 \\
\tfrac14 & \tfrac38 & \tfrac38 & 1 & 1 & 1 & 0\\
\tfrac14 & \tfrac12 & \tfrac14 & 2 & 1 & 1 & 0
\end{pmatrix}
\begin{pmatrix}
s_1 \\ s_2 \\ s_3 \\ d_1 \\ d_2 \\ d_3 \\ d_4
\end{pmatrix}
= 0.
\]
The above system has the unique solution $s_1 = s_2 = s_3 = d_1 = d_2= d_3=d_4=0$.
Therefore, $\pi$ is an extreme function.
\end{example}

\begin{table}
  \caption{A sample Sage session on the extremality test}
  \label{tab:extremality_test}
  \begin{tikzpicture}[overlay]
     \node at (12.5,-4) {\includegraphics[scale=0.5]{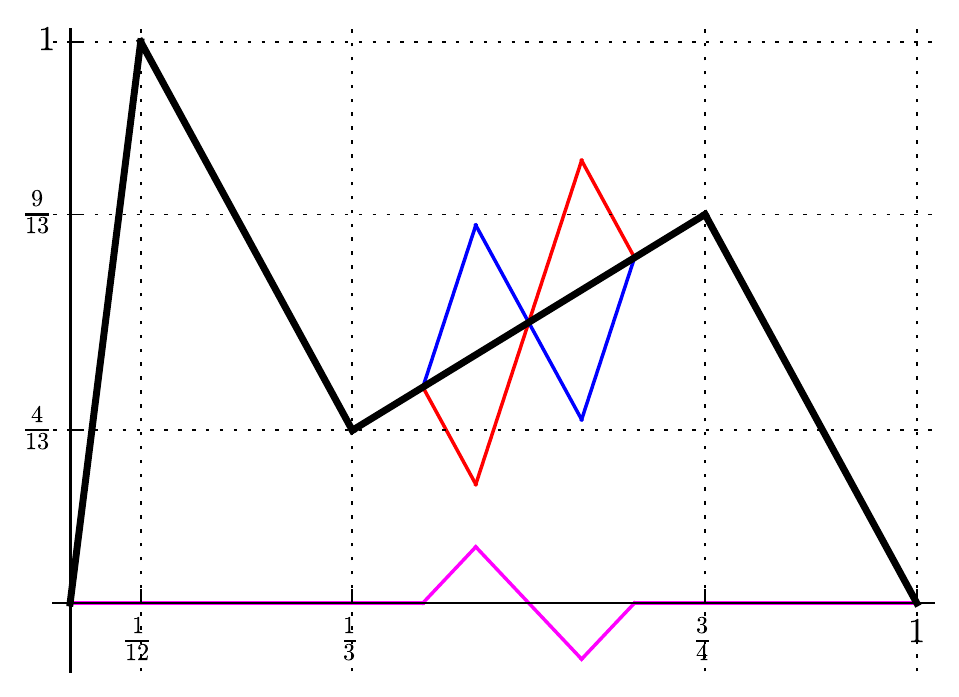}};
  \end{tikzpicture}
  \tiny 
  \begin{tabular}{@{}p{\linewidth}@{}}
    \toprule
	\begin{verbatim}
	## First we load a function and store it in variable h.
	## We start with the easiest function, the GMIC.
	sage: h = gmic()
	INFO: 2016-08-08 16:51:31,048 Rational case.
	
	## Test its extremality; this will create informative output and plots
	sage: extremality_test(h, show_plots=True)
	INFO: 2016-08-08 16:54:22,014 pi(0) = 0
	INFO: 2016-08-08 16:54:22,016 pi is subadditive.
	INFO: 2016-08-08 16:54:22,016 pi is symmetric.
	INFO: 2016-08-08 16:54:22,018 Thus pi is minimal.
	INFO: 2016-08-08 16:54:22,018 Plotting 2d diagram...
	INFO: 2016-08-08 16:54:22,018 Computing maximal additive faces...
	INFO: 2016-08-08 16:54:22,022 Computing maximal additive faces... done
	Launched png viewer for Graphics object consisting of 25 graphics primitives
	INFO: 2016-08-08 16:54:22,526 Plotting 2d diagram... done
	INFO: 2016-08-08 16:54:22,526 Computing covered intervals...
	INFO: 2016-08-08 16:54:22,527 Computing covered intervals... done
	INFO: 2016-08-08 16:54:22,527 Plotting covered intervals...
	Launched png viewer for Graphics object consisting of 2 graphics primitives
	INFO: 2016-08-08 16:54:22,985 Plotting covered intervals... done
	INFO: 2016-08-08 16:54:22,986 All intervals are covered (or connected-to-covered). 2 components.
	INFO: 2016-08-08 16:54:23,035 Finite dimensional test: Solution space has dimension 0
	INFO: 2016-08-08 16:54:23,035 Thus the function is extreme.
	
	## The docstring tells us that we can set the `f' value using an optional argument.
	sage: h = gmic(1/5)
	INFO: 2016-08-08 16:55:59,440 Rational case.
	## Of course, we know it will still be extreme; but let's test it to
	## see all the informative graphs.
	sage: extremality_test(h, show_plots=True)
	[...]
	True

	## Let's try a different function from the compendium. 
	## We change the parameters a little bit, so that they do NOT satisfy the known
	## sufficient conditions from the literature about this class of functions.
	sage: h = drlm_backward_3_slope(f=1/12, bkpt=4/12)
	INFO: 2016-08-08 17:03:20,438 Conditions for extremality are NOT satisfied.
	INFO: 2016-08-08 17:03:20,439 Rational case.
	
	## Let's run the extremality test.
	sage: extremality_test(h, show_plots=True)
	[...]
	INFO: 2016-08-08 17:03:44,796 Thus pi is minimal.
	[...]
	INFO: 2016-08-08 17:03:46,417 Uncovered intervals: ([[5/12, 2/3]],)
	[...]
	INFO: 2016-08-08 17:03:49,544 Plotting perturbation... done
	INFO: 2016-08-08 17:03:49,545 Thus the function is NOT extreme.
	False
	## Indeed, it's not extreme.  
	## We see a perturbation in magenta and the two perturbed functions in blue and red,
	## whose average is the original function (black).
	
	## Here's the Gomory fractional cut.
	sage: h = gomory_fractional()
	## It is not even minimal:
	sage: minimality_test(h, True)
	INFO: 2016-08-08 17:06:33,647 pi(0) = 0
	INFO: 2016-08-08 17:06:33,648 pi is not minimal because it does not stay in the range of [0, 1].
	False
	
	## There's many more functions to explore. Use the Tab key on the next line to see a collection of those functions.
	sage: extreme_functions.[tab]
	\end{verbatim}
    \\
	\bottomrule
  \end{tabular}
\end{table}
%
%
\section{Transformations of piecewise linear functions}
\label{sec:procedrue}
The Python module named \sage{procedures} gives access to the ``procedures'' that can transform   extreme functions.  
These procedures include the transformations between functions for finite and infinite group problems that we will discuss in \autoref{sec:finite-group}, and other well-known transformations; see \cite[Table 5]{igp_survey}. For instance, we have:
\begin{description}
\item[\sage{multiplicative\_homomorphism}] constructs the function $x \mapsto
  \pi(\lambda x)$ for a given piecewise linear function $\pi$, where $\lambda$
  is a nonzero integer; see \cite[sections 19.4.1,
  19.5.2.1]{Richard-Dey-2010:50-year-survey}.
  (The use of the word ``multiplicative homomorphism'' for this operation 
  comes from \cite{Richard-Dey-2010:50-year-survey}.)
\item[\sage{automorphism}] is a special case of the above, where $\lambda$ is
  $\pm1$, so that $x\mapsto \lambda x$ is an automorphism of the group $\R/\Z$.  
  If a different factor~$\lambda$ is provided, an error is signalled.
  Since $\lambda = -1$ gives the only nontrivial automorphism, $x \mapsto
  \pi(-x)$, this is the default.  See \cite{johnson1974group} for a
  discussion. 
\item[\sage{projected\_sequential\_merge}] performs the one-dimensional projected sequential merge, with operator $\lozenge_n^1 \, (\sage{gmic})$, on an extreme function $\pi$; see \cite{dey2}.
\end{description}
Newly discovered procedures since the publication of \cite{igp_survey} have been implemented as well. 
\begin{description}
\item[\sage{symmetric\_2\_slope\_fill\_in}] returns an extreme 2-slope function that approximates the given continuous minimal valid function $\pi$ with infinity norm distance less than $\epsilon$, under the assumption that $f$ is a rational number; see \cite{bhm:dense-2-slope}.
\item[\sage{symmetric\_2\_slope\_fill\_in\_irrational}] is a variant, proposed
  by the third author, Yuan Zhou (2015, unpublished), of the above procedure
  \sage{symmetric\_2\_slope\_fill\_in}, which removes the assumption $f\in \Q$. 
\end{description}

\section{Functionality for the finite (cyclic) group relaxation model}
\label{sec:finite-group}
Let $q$ be a positive integer and let $f \in \frac1q\Z$.  
A discrete function $\pi|_{\frac1q\Z}\colon \frac1q\Z\to\R_+$ is minimal valid
for Gomory--Johnson's finite (cyclic) group relaxation model if and only if $\pi|_{\frac1q\Z}$ is $\Z$-periodic, subadditive, $\pi|_{\frac1q\Z}(0)=0$, and satisfies the symmetry condition.
A minimal valid function $\pi|_{\frac1q\Z}$ is extreme if it is not a proper convex combination of other minimal valid functions.

In our code, a $\Z$-periodic discrete function $\pi|_{\frac1q\Z}$ can be set
up using the Python function named
\sage{discrete\_function\_from\_points\_and\_values(points, values)}, given
two finite lists \sage{points} and \sage{values} which represent the domain of
$\pi|_{\frac1q\Z}$ restricted to the interval $[0,1]$ and the function values
at these points, respectively.

Another important way to obtain a discrete function $\pi|_{\frac1q\Z}$ is to
restrict a piecewise linear function $\pi$ for the infinite group problem to
the cyclic group of given order $q$, by calling the procedure
\sage{restrict\_to\_finite\_group}. Conversely, the procedure
\sage{interpolate\_to\_infinite\_group} provides the piecewise linear
interpolation $\pi$ of a given discrete function $\pi|_{\frac1q\Z}$. 

The functionalities \sage{plot\_2d\_diagram}, \sage{minimality\_test} and
\sage{extremality\_test} that we described above also work for
$\pi|_{\frac1q\Z}$ in the finite group model.

The transformations (procedures) \sage{multiplicative\_homomorphism} and
\sage{automorphism} described in \autoref{sec:procedrue} are also available
for the finite group case.  The latter is more interesting in the finite group
case, because the cyclic group $\frac1q\Z/\Z$ has more nontrivial
automorphisms. Indeed, studying the
action of the automorphisms was a major focus of Gomory's original paper
\cite{gom}. The given factor~$\lambda$ must be an integer coprime with the
group order~$q$; see \cite[section 19.5.2.1]{Richard-Dey-2010:50-year-survey}.

See \autoref{tab:discrete-funtions} for a sample Sage session on discrete
functions for the finite group problem and the related transformations. 

\begin{table}
  \caption{A sample Sage session on discrete functions for the finite group problem.}
  \label{tab:discrete-funtions}
  \tiny 
  \begin{tabular}{@{}p{\linewidth}@{}}
    \toprule
    \begin{tikzpicture}[overlay]
     \node[align=center] at (12.5,-2) {\scriptsize\sage{plot\_with\_colored\_slopes(h)}\\\includegraphics[scale=0.6]{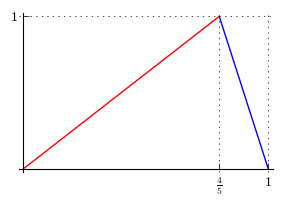}};
     \node[align=center]  at (12.5,-6) {\scriptsize \sage{plot\_with\_colored\_slopes(hr)}\\ \includegraphics[scale=0.6]{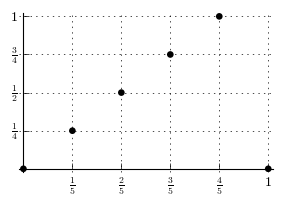}};
     \node[align=center]  at (12.5,-10) {\scriptsize\sage{plot\_with\_colored\_slopes(ha)}\\ \includegraphics[scale=0.6]{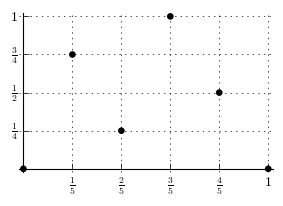}};
     \node[align=center]  at (12.5,-14) {\scriptsize \sage{plot\_with\_colored\_slopes(hi)}\\ \includegraphics[scale=0.6]{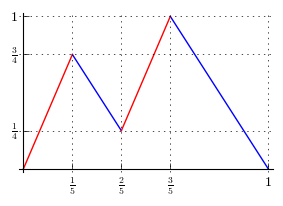}};
     \end{tikzpicture}
	\begin{verbatim}
	## We load the GMIC function with f=4/5, and store it in variable h.
	sage: h = gmic(f=4/5)
	INFO: 2016-08-08 17:23:05,206 Rational case.
	
	## We can restrict to a finite group problem.
	sage: restrict_to_finite_group?
	Signature:      restrict_to_finite_group(function, f=None, oversampling=None, order=None)
	Docstring:
	Restrict the given function to the cyclic group of given order.
	
	   If order is not given, it defaults to the group generated by the
	   breakpoints of function and f if these data are rational. However,
	   if oversampling is given, it must be a positive integer; then the
	   group generated by the breakpoints of function and f will be
	   refined by that factor.
	
	   If f is not provided, uses the one found by find_f.
	
	   Assume in the following that f and all breakpoints of function lie
	   in the cyclic group and that function is continuous.
	
	   Then the restriction is minimal valid if and only if function is minimal valid. 
	   The restriction is extreme if function is extreme. 
	
	   If, in addition oversampling >= 3, then the following holds: The
	   restriction is extreme if and only if function is extreme. This is
	   Theorem 1.5 in [IR2].
	
	   This oversampling factor of 3 is best possible, as demonstrated by
	   function kzh_2q_example_1 from [KZh2015a].
	   [...]

	sage: hr = restrict_to_finite_group(h)
	INFO: 2016-08-08 17:26:36,047 Rational breakpoints and f; 
	using group generated by them, (1/5)Z
	
	## This function can be set up by providing the breakpoints and the values.
	sage: discrete_function_from_points_and_values(points=[0, 1/5, 2/5, 3/5, 4/5, 1], 
	          values=[0, 1/4, 1/2, 3/4, 1, 0]) == hr
	INFO: 2016-08-8  17:26:37,190 Rational case.
	True

	## The restricted function is extreme for the finite group problem.
	sage: extremality_test(hr)
	INFO: 2016-08-8  17:26:38,121 pi(0) = 0
	INFO: 2016-08-8  17:26:38,123 pi is subadditive.
	INFO: 2016-08-8  17:26:38,123 pi is symmetric.
	INFO: 2016-08-8  17:26:38,124 Thus pi is minimal.
	INFO: 2016-08-8  17:26:38,124 Rational breakpoints and f; 
	using group generated by them, (1/5)Z
	INFO: 2016-08-8  17:26:38,158 Solution space has dimension 0
	INFO: 2016-08-8  17:26:38,158 Thus the function is extreme.
	True
	
	## For the finite group problems, automorphisms are interesting!
	sage: ha = automorphism(hr, 2)
	INFO: 2016-08-08 17:26:41,100 Rational breakpoints and f; 
	using group generated by them, (1/5)Z
	
	## We can interpolate to get a function for the infinite group problem.
	sage: hi = interpolate_to_infinite_group(ha)
	\end{verbatim}
    \\
    \bottomrule
  \end{tabular}
\end{table}
  
\section{Documentation and test suite}
\label{s:doc-test}

Following the standard conventions of Sage, the documentation strings contain
usage examples with their expected output.  If Sage is invoked as
\begin{quote}
  \texttt{sage -t $\langle\mathit{filename}\rangle$.sage},
\end{quote}
these examples are run and their results are compared to the expected results;
if the results differ, this is reported as a unit test failure.  This helps to
ensure the consistency and correctness of the algorithms and of the compendium of
extreme functions as we continue developing the software.  The command
\texttt{make check} runs the tests for all files of our software. 

A substantial part of our software is concerned with plotting certain
diagrams.  Our code also contains a testsuite of diagrams that have been
published in the survey \cite{igp_survey,igp_survey_part_2} and in several
research papers to ensure that these diagrams can still be reproduced with new
versions of our software.  The testsuite, invoked by \texttt{make
  check-graphics}, produces a multi-page PDF file which 
needs to be compared with a ``good'' copy of the PDF file by visual
inspection.
\smallbreak

Finally, \textsf{demo.sage} demonstrates further functionality and the use of
  the help system.

\section*{Acknowledgement(s)}

The second author wishes to thank his coauthors in parts I--IV of the present
series on the algorithmic theory of the Gomory--Johnson model, Amitabh Basu
and Robert Hildebrand.  Without this collaboration, developing the software
described in this paper would not have been possible.

An extended abstract of 8 pages appeared under the title \emph{Software for
  cut-generating functions in the Gomory--Johnson model and beyond} in Proc.\@
International Congress on Mathematical Software 2016.

\section*{Funding}

The first author's contribution was done during a Research Experience for
Undergraduates at the University of California, Davis. He was partially
supported by the National Science Foundation through grant DMS-0636297
(VIGRE).  All authors were partially supported by the National Science
Foundation through grant DMS-1320051 awarded to M.~K\"oppe.  The support is
gratefully acknowledged.

\appendix

\section{Enumerating the maximal additive faces of $\Delta\P$ in the continuous case}
\label{sec:algo-maximal-additive-face}
In this section, we explain the algorithm that the code \sage{generate\_maximal\_additive\_faces\_continuous($\pi$)} uses to enumerate the maximal additive faces for a continuous subadditive function $\pi$.
For functions with many breakpoints, it pays off to avoid duplicate
computations that would arise from the fact that a given lower-dimensional
face $F$ of $\Delta\P$ has many representations as $F(I, J, K)$ with $I, J,
K\in \P$.  
The algorithm considers the full-dimensional (two-dimensional) faces $F = F(I,
J, K)$ of $\Delta\P$ in the lexicographic order on $(I, J, K)$, where $I, J,
K$ are proper intervals of $\P$ with $I, J \subseteq [0,1]$ and $K\subseteq
[0,2]$, and $I \leq J$, i.e., $I$ is to the left of $J$ on the real line.
The algorithm has a fast path for discarding faces $F(I, J, K)$ that are not
full-dimensional. Lower-dimensional additive faces will be considered as subfaces of
full-dimensional faces as follows.  The algorithm distinguishes the following
four cases, depending on the number of additive vertices that each
full-dimensional $F$ has. 
\begin{enumerate}
\item If $F$ has no additive vertex, then $F$ does not contain any maximal additive face of $\Delta\P$ (nothing is green in $F$). No recording is needed for this face $F$.
\item If $F$ has only one additive vertex, denoted by $(x, y)$, then this
  vertex could either be a maximal additive face, or it could be a subface of
  a maximal additive face that is not a subface of $F$.  
  We maintain a hashed set of all vertices of maximal additive faces that have
  already been recorded during the algorithm. 
  We only process $(x,y)$ when $F$ is the last two-dimensional face
  $F(I,J,K)$ in the lexicographic order on $(I, J, K)$ that contains $(x,y)$.  
  This condition is easy to check; see \autoref{fig:vertex-halfopen-criterion}
  for an illustration of the cases.
  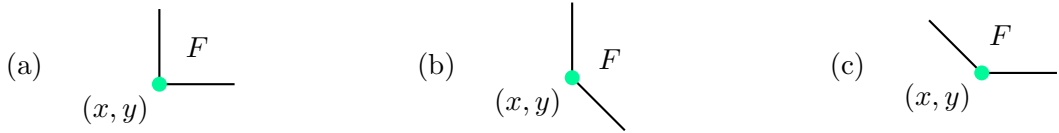
\begin{figure}[t]
    \centering
    (a)\quad\begin{minipage}{.2\textwidth}
      \begin{tikzpicture}[scale=1/2]
        \draw[](1,1) node[]{$F$}; \draw[black, thick] (0,0) node[anchor=north
        east]{$(x,y)$} -- (0,2) node[]{}; \draw[black, thick] (0,0)
        node[circle, fill=mediumspringgreen, inner sep=2pt]{} -- (2,0)
        node[]{};
      \end{tikzpicture}
    \end{minipage}
    \hfill
    (b)\quad\begin{minipage}{.2\textwidth}
      \begin{tikzpicture}[scale=1/2]
        \draw[](1,0.5) node[]{$F$}; \draw[black, thick] (0,0)
        node[anchor=north east]{$(x,y)$} -- (0,2) node[]{}; \draw[black,
        thick] (0,0) node[circle, fill=mediumspringgreen, inner sep=2pt]{} --
        (1.4,-1.4) node[]{};
      \end{tikzpicture}
    \end{minipage}
    \hfill
    (c)\quad\begin{minipage}{.2\textwidth}
      \begin{tikzpicture}[scale=1/2]
        \draw[](0.5,1) node[]{$F$}; \draw[black, thick] (0,0)
        node[anchor=north east]{$(x,y)$} -- (-1.4,1.4) node[]{}; \draw[black,
        thick] (0,0) node[circle, fill=mediumspringgreen, inner sep=2pt]{} --
        (2,0) node[]{};
      \end{tikzpicture}
    \end{minipage}
    \caption{An additive vertex $(x, y)$ is processed when it arises as a
      subface of the last two-dimensional face $F = F(I, J, K)$ in the
      lexicographic order on $(I, J, K)$ that contains $(x, y)$.
      (a) $x$ and $y$ are the left
      endpoints of $I$ and $J$; (b) $x$ and $x+y$ are the left endpoints of
      $I$ and $K$; (c) $y$ and $x+y$ are the left endpoints of $J$ and $K$.}
    \label{fig:vertex-halfopen-criterion}
  \end{figure}
In that case, we record $(x, y)$ as a maximal additive face of $\Delta\P$ when $(x,y)$ is not a vertex of a maximal additive face that has already been recorded.

\item If $F$ has exactly two additive vertices, then they form an additive
  edge $E$. It cannot lie in the interior of $F$, since otherwise
  $\Delta\pi=0$ holds over all $F$, contradicting the assumption that only two
  vertices of the two-dimensional face $F$ are additive. Therefore, this
  additive edge $E$ is a subface of $F$, and hence a one-dimensional additive
  face of $\Delta\P$.  Again we only process $E$ when $F$ is the last
  two-dimensional face $F(I,J,K)$ in the lexicographic order on $(I, J, K)$
  that contains $E$.  (There is only one other face $F'$ containing~$E$; see
  \autoref{fig:edge-halfopen-criterion}.) 
  In that case, we record $E$ if $F'$ was not recorded as an additive face.
\begin{figure}[t]
  \begin{minipage}{.3\textwidth}
    \centering
    \begin{tikzpicture}[scale=1/2]
      \draw[](1, 0) node[]{$F$}; \draw[](-1, 0) node[]{$F'$}; \draw[black,
      thick] (0,-1) node[circle, fill=mediumspringgreen, inner sep=2pt]{} --
      (0,1) node[circle, fill=green, inner sep=2pt]{};
    \end{tikzpicture}
  \end{minipage}
  \begin{minipage}{.3\textwidth}
    \centering
    \begin{tikzpicture}[scale=1/2]
      \draw[](0.5, 0.5) node[]{$F$}; \draw[](-0.5, -0.5) node[]{$F'$};
      \draw[black, thick] (-1, 1) node[circle, fill=mediumspringgreen, inner
      sep=2pt]{} -- (1,-1) node[circle, fill=green, inner sep=2pt]{};
    \end{tikzpicture}
  \end{minipage}
  \begin{minipage}{.3\textwidth}
    \centering
    \begin{tikzpicture}[scale=1/2]
      \draw[](0, 1) node[]{$F$}; \draw[](0, -1) node[]{$F'$}; \draw[black,
      thick] (-1,0) node[circle, fill=mediumspringgreen, inner sep=2pt]{} --
      (1,0) node[circle, fill=green, inner sep=2pt]{};
    \end{tikzpicture}
  \end{minipage}
  \caption{An additive edge $E = F' \cap F$ with $F' < F$ in the lexicographic
    order on $(I, J, K)$ is processed when it arises as a subface of $F$}
  \label{fig:edge-halfopen-criterion}
\end{figure}
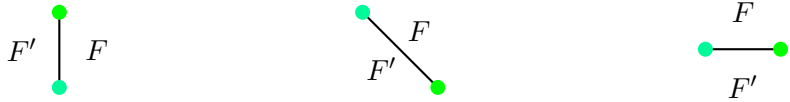
\item  If $F$ has more than $2$ additive vertices, then $\Delta\pi=0$ holds over all $F$. The face $F$ is recorded as a maximal additive face of $\Delta\P$. 
\end{enumerate}

\providecommand\ISBN{ISBN }
\bibliographystyle{gOMS-url}
\bibliography{../bib/MLFCB_bib}

\end{document}